\documentclass[12pt,leqno]{amsart}
\usepackage{amssymb,amsmath,amsthm,bm}
\usepackage{graphicx}
\usepackage{color}
\usepackage[textsize=tiny]{todonotes}
\usepackage[normalem]{ulem}
\usepackage[bookmarksopen,bookmarksdepth=3,colorlinks,citecolor=red,pagebackref,hypertexnames=true]{hyperref}
\usepackage[msc-links,nobysame,non-sorted-cites, initials]{amsrefs}
\usepackage[inline]{enumitem}
\usepackage{mathtools}
\mathtoolsset{showonlyrefs}

\definecolor{darkblue}{RGB}{0,0,160}

\definecolor{indigo}{rgb}{0.4,0,0.9}

\usepackage[lining]{libertine}   
\usepackage{cabin}
\usepackage[libertine]{newtxmath}
\usepackage[T1]{fontenc}

\def\e{{\rm e}}

\def\d{{\rm d}}
\def\dist{{\rm dist}}

\def\R {\mathbb{R}}

\def\loc {{\mathrm{loc}}}

\def\pt {{\partial_t}}

\def\1 {{\mbox{\boldmath 1}}}

\def \ps {\partial_s}

	 \def\XXint#1#2#3{{\setbox0=\hbox{$#1{#2#3}{\int}$}
	     \vcenter{\hbox{$#2#3$}}\kern-.5\wd0}}

\def \no#1#2#3 {{\bf #1} (#3), #2.}
\def \eds#1#2#3 {#1, #2, #3.}

\newtheorem{theorem}{Theorem}[section]
\newtheorem{proposition}[theorem]{Proposition}
\newtheorem{lemma}[theorem]{Lemma}
\newtheorem{corollary}[theorem]{Corollary}
\theoremstyle{definition}

\newtheorem{remark}[theorem]{Remark}
\newtheorem*{warn*}{Two  warnings}
\numberwithin{equation}{section}
\numberwithin{figure}{section}

 \makeatother

\begin{document}

\title[Coleman--Gurtin with Beltrami conductivity]{The Planar Coleman--Gurtin Model with Beltrami Conductivity}

	\author[F.  Di Plinio]{Francesco Di Plinio} 
\address[F.  Di Plinio]{Dipartimento di Matematica e Applicazioni ``R.  Caccioppoli'' \  \newline\indent Universit\`a di Napoli ``Federico II'', Via Cintia, Monte S.  Angelo 80126 Napoli, Italy}
\email{\href{mailto:francesco.diplinio@unina.it}{\textnormal{francesco.diplinio@unina.it}}}
\thanks{Research  partially supported by the FRA 2022 Program of Universit\`a degli Studi di  Napoli Federico II, project ReSinAPAS--Regularity and Singularity in Analysis, PDEs, and Applied Sciences.}
  
\subjclass[2010]{Primary: 35B41,  35K05, 30C62. Secondary: 45K05, 47H20}
\keywords{Coleman--Gurtin equation, heat conduction with memory, Dafermos history space;
Beltrami operator, anisotropic diffusion, maximal parabolic regularity;
exponential attractor, global attractor.}

\begin{abstract}
This article addresses the planar Coleman--Gurtin heat equation with memory on a bounded domain, with rough
anisotropic diffusion $A_\mu$, typical of heterogeneous or composite media and  encoded by a Beltrami coefficient $\mu\in L^\infty(\Omega)$ satisfying  $\|\mu\|_{\infty}<1$.
 First, under no additional smoothness assumptions on $\mu$, solutions with $H^1_0(\Omega)$-based initial data enter a time-averaged $L^\infty(\Omega)$   regime, and instantaneously regularize into the second-order graph space $D(A_\mu)$. 
 Assuming in addition $\mu\in W^{1,2}(\Omega)$, this regularization upgrades to $W^{2,p}(\Omega)$ for every $1<p<2$, and we construct
\emph{regular} global and exponential attractors of  finite fractal dimension, for both the $L^2(\Omega)$ and $H^1_0(\Omega)$-based dynamics. The proof combines the instantaneous smoothing method of Chekroun, Di Plinio, Glatt-Holtz and Pata with   maximal parabolic regularity for divergence-form operators with measurable coefficients, and with  planar quasiconformal Beltrami estimates recently obtained in work by Green, Wick and the author.

\end{abstract}

	\maketitle
	
\section{Introduction}

\label{sec:main}

This article is concerned with hereditary heat propagation in two-dimensional media exhibiting a complex, anisotropic and not necessarily smoothly varying internal structure, where  the standard Laplacian is replaced   with a divergence-form operator arising from the theory of quasiconformal mappings.

\subsection{Hereditary heat conduction in anisotropic media: the model}

Let \[
u=u(x,t): \Omega \times [0,\infty) \to \mathbb R\] be the temperature fluctuation field associated to the medium occupying the open, bounded region $\Omega \subset \R^2$.
The material properties of the medium are encoded by a measurable \emph{Beltrami coefficient}
$\mu\in L^\infty(\Omega)$ satisfying the uniform ellipticity condition
\[
\|\mu\|_{L^\infty(\Omega)}\le k<1.
\]
This coefficient determines a  real symmetric conductivity matrix $\sigma_\mu(x)$ via the classical
Beltrami--conductivity correspondence,
\begin{equation}\label{eq:ellipticitymat}
\sigma_\mu(x) \coloneqq \frac{1}{1-|\mu(x)|^2}
\begin{pmatrix}
|1-\mu(x)|^2 & -2\,\Im\mu(x) \\[2mm]
-2\,\Im\mu(x) & |1+\mu(x)|^2
\end{pmatrix}.
\end{equation}
For later use we introduce the constants
\[
m(k)\coloneqq\frac{1-k}{1+k},\qquad M(k)\coloneqq\frac{1+k}{1-k}.
\]
Then $\sigma_\mu$ is uniformly elliptic and bounded in the quantitative form
\begin{equation}\label{eq:ellipticity}
m(k)|\xi|^{2}\le \sigma_\mu(x)\xi\cdot\xi\le M(k)|\xi|^{2}
\qquad\text{ a.e. }x\in\Omega,\, \forall \xi\in\mathbb R^{2}.
\end{equation}
Physically, $\sigma_\mu$ represents the thermal conductivity tensor of a material whose microstructure
distorts the heat flux, twisting and stretching it in a manner prescribed by the complex dilatation $\mu$.
Thereby, define the elliptic Beltrami operator $A_\mu$ with homogeneous Dirichlet boundary conditions as
\[
A_\mu u \coloneqq -  \mathrm{div}(\sigma_\mu \nabla u), \qquad u|_{\partial \Omega}=0.
\]
In the isotropic limit $\mu \equiv 0$, the matrix reduces to the identity, and we recover the classical
operator $A=-\Delta$. Denoting by $\lambda_1(-\Delta;\Omega)$ the first Dirichlet eigenvalue of $-\Delta$,
the Rayleigh quotient and \eqref{eq:ellipticity} yield the uniform lower bound
\begin{equation}
	\label{e:lbeigen}
\lambda_1(A_\mu)\ge \underline{\lambda}(k,\Omega) \coloneqq m(k)\lambda_1(-\Delta;\Omega) 
\end{equation}
for the first eigenvalue of $A_\mu$.

The material is subject to a constitutive relation for the heat flux of  Coleman--Gurtin type. That is, the heat flux depends not only on the current temperature gradient but also on the summed past history of the gradient. This leads to the  integro-differential semilinear  equation \begin{equation}\label{eq:intro-model}
    \partial_t u(\cdot, t) + A_\mu u(t) + \int_0^\infty \kappa(s) A_\mu u(\cdot, t-s)  \,\d s + \phi(u(\cdot, t)) = f, \qquad t>0
\end{equation}
supplemented with the initial condition
\begin{equation}
	\label{e:intro-initial}
	u(x,t) = u_0(x,t), \qquad x\in \Omega,\, t\leq 0,
\end{equation}
where the given function $u_0:\Omega \times (-\infty, 0]\to \mathbb R$ accounts for the \emph{initial past history} of $u$.
Here, $\kappa$ is a nonnegative summable memory kernel characterizing the rate at which the material forgets past thermal states. The nonlinear term $\phi(u)$ represents a temperature-dependent reaction or radiative loss.
Finally,  $f$ is an  external heat supply which is assumed to be time-independent.

\subsection{Assumptions on the nonlinearities and the memory kernel}
 We assume the domain $\Omega \subset \mathbb{R}^2$ is bounded and $\partial \Omega$ is of class  $\mathcal{C}^2$. The structural assumptions on the parameters and nonlinear terms are as follows.

\subsubsection*{The Beltrami coefficient}
The dilatation $\mu:\Omega \to \mathbb C$ satisfies the ellipticity condition \begin{equation} \label{e:ellipticitymu}
	\|\mu\|_{L^\infty(\Omega)} \le k < 1.
\end{equation}Assumption \eqref{e:ellipticitymu} is standing throughout the article. The regularity hypothesis
 \begin{equation} \label{e:regmu}
\mu \in W^{1,p_0}(\Omega), \qquad p_0\geq 2
\end{equation}
will  only be imposed when strictly necessary.

\subsubsection*{The memory kernel}
  Assume $\kappa:[0,\infty) \to (0,\infty)$ is absolutely continuous and introduce the \emph{memory kernel generator}\footnote{The memory kernel generator is almost invariably denoted by the letter $\mu$ in past literature on the subject, see for instance \cites{CGDP2010,AMN}. We choose to reserve the symbol $\mu$ to the Beltrami coefficient and use $h$ instead.} $h$ via the relation
\[
\kappa(s) = \kappa_0 - \int_0^s h(\tau) \, \d \tau.
\]
We impose the following conditions on $h$.
\begin{itemize}

\item[(K1)] The function $h: \mathbb{R}^+ \to \mathbb{R}^+$ is non-negative, non-increasing, and summable with total mass
\[
\int_0^\infty h(s) \, \d s = \kappa_0 > 0.
\]
Consequently, $\kappa$ is non-negative, non-increasing, and summable. We fix the normalization
\begin{equation}
	\label{e:kappanorm}
\int_0^\infty \kappa(s) \, \d s = \int_0^\infty s h(s) \, \d s = 1.
\end{equation}
\item[(K2)] (Generalized dissipation) There exists a constant $\Theta > 0$ such that
\[
\kappa(s) \le \Theta h(s) \quad \text{for a.e. } s > 0.
\]
Consequently, $\kappa $ exhibits the exponential decay property
\[
\kappa(s) \leq \kappa_0 \exp\left({-\frac{s}{\Theta}}\right), \qquad s\geq 0.
\]
\end{itemize}
\subsubsection*{External force and nonlinearities}
\begin{itemize}
\item[(F)] The time-independent forcing satisfies $f \in L^2(\Omega)$.
\item[(P1)] (Dissipativity) Let $\phi \in \mathcal{C}^3(\mathbb{R})$ with $\phi(0) =0$. Then, with reference to \eqref{e:lbeigen}
\[
\liminf_{|u| \to \infty } \phi'(u) > -\underline{\lambda}(k,\Omega).
\]
\item[(P2)] (Polynomial growth) There exist $C > 0$ and $m \ge 1$ such that
\[
\left|\phi^{(1+j)}(r)\right| \le C(1+|r|^{m-j}), \qquad \forall r \in \mathbb{R},\, j \in \{0,1,2\}.
\]
\end{itemize}

\subsection{Remark on Trudinger--Moser alternatives}
Although, for the sake of simplicity, in \textbf{(P2)}  we require arbitrary polynomial growth,
 the analysis and the elliptic estimates that follow are built so that  a weaker Trudinger--Moser type  assumption may be imposed instead, with rather minimal changes. This task is possible because of the   exponential in $\frac{1}{p_0-p}$ blow-up rate of the elliptic estimate of Proposition \ref{prop:W2p-A} below, which is actually made explicit in the reference \cite{DPF}, and is    left  to the interested reader.

\subsection{Physical motivation}
Models of heat propagation with \emph{memory} account for relaxation effects that are not captured by
Fourier's law. In loose parlance, the heat flux at time $t$ depends on a suitable weighting of the  \emph{history} of the temperature gradient,
encoding fading-memory thermodynamics. This viewpoint
originates in the continuum thermodynamics program of Coleman--Gurtin,  
  and the subsequent theory of heat conductors with memory developed by Gurtin--Pipkin and
others, see, e.g.\ \cites{CGZAMP67,GP68,Nun71}.

In heterogeneous or composite planar media, the relevant conductivity is often anisotropic and may
vary abruptly across microscopic interfaces.  In two dimensions, such anisotropy admits a particularly
natural geometric parametrization. Namely, uniformly elliptic planar conductivities are encoded, up to a
normalization by a \emph{Beltrami coefficient} $\mu$, i.e.\ by the complex dilatation of a quasiconformal
distortion map rectifying the   flux lines.  This  correspondence is a standard
bridge between divergence form diffusion, quasiconformal mapping theory, and planar conductivity
problems, including inverse problems of  Calder\'on type. See   \cites{AIM,AP06} for systematic treatments.
From this perspective, \eqref{eq:intro-model} describes hereditary heat conduction in a planar medium
whose microscopic anisotropy is allowed to be in principle merely measurable under the uniform
ellipticity constraint. A prominent physical realization of this setting is found in the thermal analysis of thin fiber-reinforced composite laminates and woven polymeric nanocomposites. In such materials, the complex microscopic interfaces induce highly anisotropic, abruptly varying planar heat fluxes. These are perfectly captured by a measurable Beltrami coefficient, while the polymer matrix inherently exhibits fading-memory thermodynamics that necessitate a hereditary constitutive law, see Milton \cite{Milt} for a broader viewpoint.

\subsection{Main results and technical novelties}\label{subsec:main-results}

The aim of this paper is to prove that the hereditary Beltrami heat equation \eqref{eq:intro-model},
although driven by a \emph{rough} anisotropic conductivity encoded by a merely measurable Beltrami
coefficient $\mu$ satisfying \eqref{e:ellipticitymu}, generates a strongly dissipative semigroup whose
asymptotic behavior is captured by finite-dimensional, regular attractors. Standard references on global attractors for semigroups of Mathematical Physics are the monographs \cites{CV,HAL,TEM}. A classical account of the theory of  exponential attractors  can be found in \cite{EFNT}, see also the survey \cite{MZ} and the more recent contribution \cite{Z2023} by Zelik. Attractors for equations with memory have been at the forefront of the analysis of dissipative partial differential equations: an extremely partial list of contributions, biased towards the topic at hand, is \cites{CDP,DaneseGeredeliPata2015DCDS,DellOroMammeri2021AMO, DiPlinioGiorginiPataTemam2018JNS,GattiGrasselliMiranvillePataSquassina2005DCDS,GGMP2010, GattiGrasselliPata2004PhysD,GrasselliPata2005JEE,KloedenRealSun2011CPAA,LiYang2023JDDE, Shomberg2020TMNA}.

In our case, the  physical interpretation is that the fading-memory kernel $\kappa$ enforces an
irreversible relaxation. After a transient depending on the initial past history \eqref{e:intro-initial},
solutions enter a uniformly controlled regime and the subsequent evolution is confined to a compact set
of effective degrees of freedom, robust under perturbations and exponentially attracting. 

In the three-dimensional Euclidean setting, the works \cites{CGDP2010,AMN} exploit a partial instantaneous
regularization on the $H^1_0(\Omega)$-based phase space. Specifically, the dynamics enters $H^2(\Omega)\cap H^1_0(\Omega)$
for positive times, thereby unlocking Agmon-type inequalities to handle the semilinear term $\phi$.
This mechanism traces back to the strongly damped wave equation \cite{PZ} and related integro-differential
models \cite{DP1,DP2,DPZ}.
In the present two-dimensional anisotropic setting, the overall strategy shares several steps with \cite{CGDP2010}, but the rough Beltrami conductivity prevents any direct use of second derivative norms. 

Therefore, the first necessary novelty is   a different route to uniform control. Namely, we show that the
$H^1_0$-based dynamics enters an $L^\infty$-bounded absorbing set by combining \emph{maximal parabolic regularity}
for divergence-form operators with measurable coefficients \cites{HR,KunstmannWeis2004} with the sharp planar
Beltrami $W^{1,q}$ theory from quasiconformal analysis \cite{AIM}, see also \cite{DPF}.
This $L^\infty$ bound then yields instantaneous smoothing into the second-order graph space $\mathcal V$,
see Theorem~\ref{thm-inst}, and requires no further assumptions on $\mu$ beyond \eqref{e:ellipticitymu}.

A finer description of the asymptotic dynamics is obtained under the mild additional hypothesis
$\mu\in W^{1,2}(\Omega)$. Although the associated conductivities $\sigma_\mu$ may still be far from Lipschitz,
this Sobolev regularity allows us to use the quasiconformal elliptic estimates established in \cite{DPF} by the author and collaborators.
These are essential to control nonlinear difference terms in the continuous-dependence and squeezing
estimates underlying the construction of exponential attractors, leading to Theorem~\ref{thm:main-exp-attractor-V}.
In particular, this implies the existence of suitably regular global attractors for the semigroup on both the
$L^{2}$-based and $H^{1}_{0}$-based phase spaces,  see Corollary~\ref{cor:main-global-attractors}.
\label{Ss15}
\subsection{Structure of the paper}
Section~\ref{sec:functional-setting} sets the operator-theoretic framework for Beltrami diffusion, introduces
the past-history formulation and the memory spaces, and collects the elliptic and parabolic tools used later.
Section~\ref{sec:mainres} states the main results, from well-posedness,  dissipativity and smoothing for merely measurable
conductivities to the existence of regular finite-dimensional attractors under $\mu\in W^{1,2}(\Omega)$.
Section~\ref{sec:dissipativity-Hj} proves the semigroup properties and the dissipative estimates.
The $L^\infty$-absorbing regime for the $\mathcal H^{1}$-dynamics, obtained via maximal parabolic regularity and
planar Beltrami $W^{1,q}$ estimates, is established in Section \ref{s5}.
With this bound in hand, Section~\ref{sec:aux-Vp} derives instantaneous regularization into $\mathcal V$ and develops
the continuous-dependence estimates. When $\mu\in W^{1,2}(\Omega)$, these provide the squeezing mechanism required for
exponential attractors. Finally, Sections~\ref{sec:reg-exp-attr} and~\ref{sec:exp-attractor-V} complete the construction
of regular exponentially attracting sets and prove Theorem~\ref{thm:main-exp-attractor-V}, yielding in particular the
existence and regularity of the global attractor, cf.\ Corollary~\ref{cor:main-global-attractors}.

\section{Functional setting and elliptic toolbox}\label{sec:functional-setting}

This section contains the functional setting for the semigroup generated by \eqref{e:intro-initial} on suitably constructed phase spaces. It also serves us to recall, and adapt to our setting, the elliptic estimates on the Beltrami operator $A_\mu$ which appear in our regularization scheme. These estimates are part of the main results of \cite{DPF} by one of us and coauthors.
\begin{warn*}
Without explicit mention, we will perform certain multiplications of the relevant dynamical equations which \emph{a priori} are only formally justified provided the solution is regular enough. However, each differential inequality we thus achieve may be rigorously justified within a suitable standard Galerkin approximation scheme.

Although the Beltrami coefficient $\mu$ is in general complex-valued, the associated
conductivity matrix $\sigma_\mu$ is real symmetric by its explicit
definition \eqref{eq:ellipticitymat}. In particular, the operator
$A_\mu=-\mathrm{div}(\sigma_\mu\nabla\cdot)$ has real coefficients. 
Hence, for real initial data and past history, and real forcing, the unique solution remains
real for all $t\ge0$. For this reason, we work on real Hilbert spaces
throughout. \end{warn*}

\subsection{Operator-theoretic setting for the Beltrami diffusion}\label{subsec:operator-setting}
 
Set $V^0\coloneqq L^{2}(\Omega)$ with standard inner product \[
\langle u,v \rangle_{V^0}\coloneqq \int_\Omega u(x)  {v(x)} \,\d x\] and define  the bilinear form
\begin{equation}\label{eq:form-a-mu}
a_{\mu}(u,v)\coloneqq\int_{\Omega}\sigma_{\mu}\nabla u\cdot \nabla v\,\d x,
\qquad u,v\in H^{1}_{0}(\Omega),
\end{equation}
where $\sigma_{\mu}$ is the conductivity matrix introduced in
\eqref{eq:ellipticity}. The quantitative ellipticity \eqref{eq:ellipticity} yields
\begin{equation}\label{eq:form-bounds}
m(k)\|\nabla u\|_{L^{2}}^{2}\le a_{\mu}(u,u)\le M(k)\|\nabla u\|_{L^{2}}^{2},
\qquad \forall u\in H^{1}_{0}(\Omega).
\end{equation}
Denote by $A_{\mu}$ the positive Dirichlet realization associated with $a_{\mu}$, with
\[
D(A_{\mu})\coloneqq\left\{u\in H^{1}_{0}(\Omega): \exists f\in {V^0}\ \text{s.t.}\
a_{\mu}(u,\varphi)=\langle f,\varphi\rangle_{{V^0}}\ \forall \varphi\in H^{1}_{0}(\Omega)\right\},
\quad A_{\mu}u\coloneqq f.
\]
Then $A_{\mu}$ is self-adjoint on ${V^0}$, $A_{\mu}\ge0$, and has compact resolvent.
In particular, $-A_{\mu}$ generates a strongly continuous analytic semigroup of contractions
on ${V^0}$.
For $r\geq 0$ we define the Hilbert space  ${V^r}$ by its inner product, 
\[V^r \coloneqq D\left(A_{\mu}^{{\frac{r}{2}}}\right), \quad \left\langle  u, v\right\rangle_{V^r}\coloneqq
\left\langle A_\mu^{{\frac{r}{2}}}u,A_\mu^{{\frac{r}{2}}}v\right\rangle_{V^0},
\]
and set ${V^{-r}}\coloneqq (V^r)'$ with duality pairing
$\langle\cdot,\cdot\rangle_{-r,r}$.
By the spectral calculus, $A_\mu^{-{\frac{r}{2}}}:{V^{-r}}\to V^0$ is an isometric isomorphism,
and we identify ${V^{-r}}$ with $D(A_\mu^{-{\frac{r}{2}}})$ via
\begin{equation}\label{eq:duality-Hminus-Hplus}
\langle u,v\rangle_{-r,r}
= \left\langle A_\mu^{-{\frac{r}{2}}}u,A_\mu^{{\frac{r}{2}}}v\right\rangle_{V^0},
\qquad u\in {V^{-r}},\ v\in V^r.
\end{equation}
Notice that, by \eqref{eq:form-bounds}, ${V^1}=D(A_{\mu}^{1/2})$ coincides with
$H^{1}_{0}(\Omega)$ and the norm $\|u\|_{{V^1}}$ is equivalent to $\|\nabla u\|_{L^{2}}$,
with constants depending only on $k$.
Let $\lambda_{1}(-\Delta;\Omega)$ denote the first Dirichlet eigenvalue of $-\Delta$.
By \eqref{eq:form-bounds} and the Rayleigh quotient,
\begin{equation}\label{eq:uniform-coercivity}
\lambda_{1}(A_{\mu})\ge \underline{\lambda}(k,\Omega)\coloneqq m(k)\,\lambda_{1}(-\Delta;\Omega),
\end{equation}
which coincides with \eqref{e:lbeigen}. In particular,
\begin{equation}\label{eq:poincare-Hscale}
\|u\|_{{V^0}}\le \underline{\lambda}(k,\Omega)^{-1/2}\,\|u\|_{{V^1}},
\qquad \forall u\in {V^1},
\end{equation}
and the constant in \eqref{eq:poincare-Hscale} is uniform over all $\mu$ with
$\|\mu\|_{\infty}\le k<1$.

\subsection{Memory spaces and past history formulation}\label{subsec:memory-spaces} Hereafter, we recall the functional setting referred to as the  \emph{Dafermos history space}, rooted in the seminal work \cite{DAF}. This framework  makes it possible to phrase  integro-differential systems such as \eqref{eq:intro-model} into an autonomous system on a suitably extended phase space. 
\begin{remark} The powerful alternative setup of \emph{memory states} has been devised more recently by Fabrizio, Giorgi and Pata in \cite{FGP}. Formulating the results and proofs of the present paper within the framework of \cite{FGP} would be of independent interest, but is not pursued here. \end{remark}

For $r\in \mathbb R$, define the weighted memory spaces
\begin{equation}\label{eq:Mr-def}
{\mathcal M^{r}}\coloneqq L^{2}_{h}(\R^{+};{V^r})
=\left\{\eta:\R^{+}\to {V^r}: \|\eta\|^{2}_{{\mathcal M^{r}}}
\coloneqq\int_{0}^{\infty}h(s)\|\eta(s)\|^{2}_{{V^r}}\,\d s<\infty\right\},
\end{equation}
endowed with the Hilbert inner product
\[\langle\eta,\psi\rangle_{{\mathcal M^{r}}}
=\int_{0}^{\infty}h(s)\langle\eta(s),\psi(s)\rangle_{{V^r}}\,\d s.\]
By \textbf{(K1)} we have the basic bound
\begin{equation}\label{eq:memory-average-bound}
\left\|\int_{0}^{\infty}h(s)\eta(s)\,\d s\right\|_{{V^r}}
\le \kappa_{0}^{1/2}\,\|\eta\|_{{\mathcal M^{r}}}.
\end{equation}
Moreover, by \textbf{(K2)} and the normalization \eqref{e:kappanorm},
\begin{equation}\label{eq:kappa-vs-h}
\int_{0}^{\infty}\kappa(s)\|\eta(s)\|_{{V^r}}^{2}\,\d s
\le \Theta \int_{0}^{\infty}h(s)\|\eta(s)\|_{{V^r}}^{2}\,\d s
=\Theta\,\|\eta\|_{{\mathcal M^{r}}}^{2}.
\end{equation}
Define the operators
\[
T_r\eta\coloneqq -\ps\eta,
\qquad
D(T_r)\coloneqq\left\{\eta\in{\mathcal M^{r}}:\ \ps\eta\in{\mathcal M^{r}},\ \eta(0)=0\right\}, \qquad r\in \mathbb R
\]
where $\eta(0)$ is understood as the trace in ${V^r}$. For simplicity, we omit the subscript $r$ from $T_r$ and simply write $T$ when the domain is clear from context.

\begin{lemma} \label{lem:translation-dissipation}
Assume \textbf{(K1)}. For every $r\in\mathbb R$ and every $\eta\in D(T_r)$,
\begin{equation}\label{eq:translation-diss}
\langle T\eta,\eta\rangle_{{\mathcal M^{r}}}
=\frac12\int_{0}^{\infty}h'(s)\|\eta(s)\|_{{V^r}}^{2}\,\d s\le 0.
\end{equation}
Consequently, $T$ generates a strongly continuous semigroup of contractions on ${\mathcal M^{r}}$,
given explicitly by the right-translation
\[
\left({\e^{tT}}\eta\right)(s)=
\begin{cases}
0, & 0<s\le t,\\
\eta(s-t), & s>t.
\end{cases}
\]
\end{lemma}

\begin{proof}
Let $\eta\in D(T_r)$. Using $\eta(0)=0$ and integrating by parts,
\[
\begin{split}
\langle T\eta,\eta\rangle_{{\mathcal M^{r}}}
&=-\int_{0}^{\infty}h(s)\langle\ps\eta(s),\eta(s)\rangle_{{V^r}}\,\d s
=-\frac12\int_{0}^{\infty}h(s)\ps\|\eta(s)\|_{{V^r}}^{2}\,\d s \\ &
=\frac12\int_{0}^{\infty}h'(s)\|\eta(s)\|_{{V^r}}^{2}\,\d s.
\end{split}
\]
Since $h$ is nonincreasing by \textbf{(K1)}, $h'\le 0$ a.e., yielding \eqref{eq:translation-diss}.
The explicit representation is the standard formula for the right-translation semigroup.
\end{proof}
Given $u\in L^{1}_{\loc}([0,\infty);{V^r})$ and $\eta_{0}\in{\mathcal M^{r}}$, the linear problem
\[
\pt\eta=T\eta+u,\qquad \eta(0)=\eta_{0},
\]
admits a unique mild solution $\eta\in C([0,\infty);{\mathcal M^{r}})$ with variation-of-constants
formula
\begin{equation}\label{eq:VOC-history}
\eta^{t}=\e^{tT}\eta_{0}+\int_{0}^{t}\e^{(t-\tau)T}u(\tau)\,\d\tau.
\end{equation}
Equivalently, one has the explicit representation, cf.\ \cite[(3.3)]{CGDP2010} 
\begin{equation}\label{eq:eta-representation}
\eta^{t}(s)=
\begin{cases}
\displaystyle\int_{0}^{s}u(t-y)\,\ dy, & 0\le s\le t,\\[2mm]
\eta_{0}(s-t)+\displaystyle\int_{0}^{t}u(t-y)\,\ dy, & s>t.
\end{cases}
\end{equation}
Given a sufficiently regular trajectory $u:\mathbb R \to V^r $, the Dafermos integrated  past history variable is defined  as
\begin{equation}\label{eq:dafermos-variable}
\eta^{t}(s)\coloneqq\int_{0}^{s}u(t-y)\,\d y,\qquad s\ge0.
\end{equation}
Then $\ps\eta^{t}(s)=u(t-s)$ and, since $\kappa'(s)=-h(s)$ by definition, formal
integration by parts gives the identity
\begin{equation}\label{eq:memory-ibp}
\int_{0}^{\infty}\kappa(s)A_{\mu}u(t-s)\,\d s
=\int_{0}^{\infty}\kappa(s)A_{\mu}\ps\eta^{t}(s)\,\d s
=\int_{0}^{\infty}h(s)A_{\mu}\eta^{t}(s)\,\d s,
\end{equation}
where boundary terms vanish because $\eta^{t}(0)=0$ and $\kappa(s)\to0$ as $s\to\infty$.
Accordingly, \eqref{eq:intro-model} is equivalently rewritten, in the appropriate weak sense, as
the Cauchy problem in $(u,\eta)$:
\begin{equation}\label{eq:CG-QC} \tag{CGB}
\begin{cases}
\pt u + A_{\mu}u + \displaystyle\int_{0}^{\infty}h(s)A_{\mu}\eta(s)\,\d s + \phi(u)
= f \\[2mm]
\pt \eta = T\eta + u,\\[1mm]
(u(0),\eta^{0})=(u_{0},\eta_{0}).
\end{cases}
\end{equation}
The consistency with the prescribed past history \eqref{e:intro-initial} is ensured by choosing
\begin{equation}\label{eq:initial-history}
u_{0}=u(\cdot,0),\qquad \eta_{0}(s)=\int_{0}^{s}u_{0}(\cdot,-y)\,\d y.
\end{equation}
In agreement with the notation used in Section~\ref{sec:main},  set
\begin{equation}\label{eq:phase-spaces}
\mathcal H^{j}\coloneqq {V^j}\times{\mathcal M^{j+1}}, \;j\in \mathbb N,
\qquad
\mathcal V\coloneqq {V^2}\times{\mathcal M^{2}}.
\end{equation}
We will mostly use indices $j\in \{0,1\}.$ The next lemma is standard when dealing with integro-differential systems with memory .
 
\begin{lemma}[\cite{CGDP2010}]\label{lem:memory-tail} 
Let $r\in \mathbb Z$ and suppose $\eta$ solves
\[
\pt\eta = T\eta + u,\qquad \eta(0)=\eta_0\in{\mathcal M^{r+1}},
\]
with
$
u\in L^\infty_{\loc}([0,\infty);V^r)\cap L^2_{h}([0,\infty);{V^{r+1}}).
$
Define 
\begin{equation}\label{eq:U-tail-def}
\mathcal U_{r+1}[\eta]\coloneqq \frac12 \int_0^\infty \kappa(s)\|\eta(s)\|_{{V^{r+1}}}^2\,\d s, \qquad  
0\le \mathcal U_{r+1}[\eta]\le \frac{\Theta}{2}\|\eta\|_{{\mathcal M^{r+1}}}^2.
\end{equation}
 Then, for a.e.\ $t>0$,
\begin{equation}\label{eq:U-tail-identity}
\frac{\d}{\d t}\mathcal U_{r+1}[\eta^t] + \frac12\|\eta^t\|_{{\mathcal M^{r+1}}}^2
= \int_0^\infty \kappa(s)\,\langle u(t),\eta^t(s)\rangle_{{V^{r+1}}}\,\d s.
\end{equation}

\end{lemma}

\begin{remark} \label{cor:CG-energy}
 Let $r\ge0$ and fix $\nu>0$. Define
\[
\mathcal{E}_{r,\nu}(u,\eta)
\coloneqq \|u\|_{V^r}^2 + \|\eta\|_{{\mathcal M^{r+1}}}^2 + 2\nu\mathcal U_{r+1}[\eta].
\]
Then $\mathcal E_{r,\nu}$ is equivalent to the $\mathcal H^r$-norm:
\begin{equation}\label{eq:E1-equiv-cor}
\|(u,\eta)\|_{\mathcal{H}^r}^2
\le \mathcal{E}_{r,\nu}(u,\eta)
\le \left(1+2\nu\Theta\right)\,\|(u,\eta)\|_{\mathcal{H}^r}^2,
\qquad (u,\eta)\in\mathcal H^r.
\end{equation}
Moreover, let $(u,\eta)$ be a sufficiently regular solution to the Cauchy problem in $(u,\eta)$
\begin{equation}\label{eq:CG-QCext}
\begin{cases}
\pt u + A_{\mu}u + \displaystyle\int_{0}^{\infty}h(s)A_{\mu}\eta(s)\,\d s  
=  g \\[2mm]
\pt \eta = T\eta + u,\\[1mm]
(u(0),\eta^{0})=\xi \in \mathcal H^r
\end{cases}
\end{equation}
 with forcing $g\in L^2_{\mathrm{loc}}([0,\infty);V^r)$.
Then, for a.e.\ $t\ge0$,
\begin{equation}\label{eq:E1-diff-ineq-raw}\begin{split}
& \quad
\frac12\frac{\d}{\d t}\mathcal E_{r,\nu}\left(u(t),\eta^t\right)
+\|u(t)\|_{{V^{r+1}}}^2
+\nu \|\eta^t\|_{{\mathcal M^{r+1}}}^2\\ &=
\langle g(t),u(t)\rangle_{V^r} + 2\nu \int_0^\infty \kappa(s)\langle u(t),\eta^t(s)\rangle_{{V^{r+1}}}\,\d s.
\end{split}
\end{equation}
 In particular, using H\"older's and  Poincar\'e's inequalities together with assumption \textbf{(K2)} entails the existence of $\nu_0>0$  depending only on parameters, such that for all $0<\nu<\nu_0$
\begin{equation}\label{eq:F-damped}
\frac{\d}{\d t}\mathcal E_{r,\nu}\left(u(t),\eta^t\right) + (1-\nu) \|u(t)\|_{{V^{r+1}}}^2 + \nu^2  \mathcal E_{r,\nu}\left(u(t),\eta^t\right)
\le 2 \langle g(t),u(t)\rangle_{V^r}.
\end{equation}
 Inequality \eqref{eq:F-damped} will be repeatedly invoked throughout with different instances of forcing $g$.

\end{remark}
 
\subsection{Beltrami elliptic toolbox}\label{sec:elliptic-toolbox} 

In what follows, the constants $C$ and those implied by the almost inequality sign $\lesssim $ may change from line to line.
Unless stated otherwise, their dependence is restricted to $(k,\Omega)$ and the
relevant exponents, and indicated as a subscript when necessary. The two results below isolate exactly where the stronger
assumption \eqref{e:regmu} is used. 

\begin{proposition} \label{prop:beltrami-W1q}
  Assume \eqref{e:ellipticitymu} holds and 
\begin{itemize}
\item[\textup{(i)}] either $2<q< q_\ast(k)\coloneqq 1+\frac1k $, or
\item[\textup{(ii)}] $\mu$ satisfies \eqref{e:regmu} for some $p_0\ge2$  and $2<q<\infty$ is arbitrary.
\end{itemize}
Let $F\in L^{q}(\Omega;\R^{2})$ and $G\in L^{q}(\Omega)$. Then $\Psi= \mathrm{div}F+G \in V^{-1}\cap W^{-1,q}(\Omega)$, and the equation
\begin{equation}\label{eq:A-div-FG}
A_{\mu}u=\Psi\qquad \text{in }V^{-1}
\end{equation}
has a unique solution $u\in V^1 \cap  W^{1,q}_{0}(\Omega) $ and
\begin{equation}\label{eq:W1q-toolbox}
\|\nabla u\|_{L^{q}(\Omega)}
\le C_{q,k,\Omega}
\left(\|F\|_{L^{q}(\Omega)}+\|G\|_{L^{q}(\Omega)}+\|u\|_{L^{2}(\Omega)}\right).
\end{equation}
 \end{proposition}

\begin{proof}
Existence and uniqueness of $u\in V^1\sim H^{1}_{0}(\Omega)$ follow from Lax--Milgram applied to the
coercive form \eqref{eq:form-a-mu} and the fact that $\Psi\in V^{-1}$.

To prove the $W^{1,q}$ gain, first reduce \eqref{eq:A-div-FG} to a pure divergence right-hand side.
Let $w\in W^{2,q}(\Omega)\cap W^{1,q}_{0}(\Omega)$ be the unique solution of $-\Delta w=G$ in $\Omega$.
Then $\nabla w\in L^{q}(\Omega)$ with $\|\nabla w\|_{L^{q}}\le C_{q,\Omega}\|G\|_{L^{q}}$ and
\[
\Psi=
\mathrm{div}F+G=\mathrm{div}\left(F-\nabla w\right)\qquad\text{in }V^{-1}\cap W^{-1,q}(\Omega).\]
Setting $\widetilde F\coloneqq F-\nabla w$, we thus have
\begin{equation}\label{eq:A-div-Ftilde}
A_{\mu}u=\mathrm{div}\widetilde F\qquad\text{in }V^{-1}\cap W^{-1,q}(\Omega)
\end{equation}
together with $\|\widetilde F\|_{L^{q}}\le \|F\|_{L^{q}}+C_{q,\Omega}\|G\|_{L^{q}}$.
Next, recall that  the Dirichlet problem
\eqref{eq:A-div-Ftilde} can be rewritten as a first-order inhomogeneous Beltrami system for a complex
potential $f=u+iv$, whose complex derivative $g\coloneqq \partial f$ satisfies a resolvent equation of the form
\begin{equation}\label{eq:beltrami-resolvent-form}
(I-\mu S_\Omega)g=\Phi\eqcolon \frac12\left((\widetilde F_1-i\widetilde F_2)+\mu(\widetilde F_1+i\widetilde F_2)\right)\qquad\text{in } L^q(\Omega)
\end{equation}
where $S_\Omega$ is the compressed Beurling transform on $\Omega$. Note that 
$\Phi$ depends linearly on $\widetilde F$, in particular $\|\Phi\|_{L^{q}(\Omega)}\lesssim_{q,\Omega}\|\widetilde F\|_{L^{q}(\Omega)}$; see, for instance, \cite[Ch.~16]{AIM} or \cite[\S3.2]{DPF} for more details.
Since $\nabla u$ is a real-linear combination of $\partial f$ and $\bar\partial f$,  it follows that
\begin{equation}\label{eq:grad-u-by-g}
\|\nabla u\|_{L^{q}(\Omega)}\le C_{q,k,\Omega}\left(\|g\|_{L^{q}(\Omega)}+\|\widetilde F\|_{L^{q}(\Omega)}+\|u\|_{L^{2}(\Omega)}\right).
\end{equation}
Thus it remains to bound $g$ in $L^{q}$ through \eqref{eq:beltrami-resolvent-form}.
In case \textup{(i)}, the classical $L^{q}$ invertibility of the Beltrami resolvent in the critical interval
implies that $I-\mu S$, and hence $I-\mu S_\Omega$, is invertible on $L^{q}$ for $2<q<q_\ast(k)$, with norm bound depending
only on $(q,k)$; see \cite[Theorem~14.0.4]{AIM} or the more precise form \cite[Lemma 2.3]{DPF}. In case \textup{(ii)}, since $\mu\in W^{1,p_0}(\Omega)$ with $p_0\ge2$,
one has $\mu\in\mathrm{VMO}(\Omega)$, hence $I-\mu S$ is invertible on $L^{q}$ for all $1<q<\infty$ by
\cite[Theorem~14.5.2]{AIM}. Also note that  quantitative bounds of the form
$\|(I-\mu S_\Omega)^{-1}\|_{L^{q}\to L^{q}}\le C_{q,k,\Omega}E_q(\|\mu\|_{W^{1,p_0}(\Omega)})$ follow from the resolvent
estimates in \cite[Lemma~2.3 and Theorem~A]{DPF} and the domain reduction discussed therein, with $E_q$ of exponential type.
Applying $(I-\mu S_\Omega)^{-1}$ to \eqref{eq:beltrami-resolvent-form} yields
$\|g\|_{L^{q}}\le C_{q,k,\Omega} \|\Phi\|_{L^{q}}$, which combined with
\eqref{eq:grad-u-by-g} and the estimate on $\widetilde F$ gives \eqref{eq:W1q-toolbox}.
\end{proof}

\begin{proposition} \label{prop:W2p-A}
Assume \eqref{e:ellipticitymu} holds. Also assume that \eqref{e:regmu} holds for some $p_0\geq 2.$ 
Let $p \in (1,p_0)$, $g\in L^{p}(\Omega)$ and let $u\in H^{1}_{0}(\Omega)$ be the (unique) weak solution of
\[
A_\mu u = g \qquad\text{in }H^{-1}(\Omega).
\]
Then $u\in W^{2,p}(\Omega)\cap H^{1}_{0}(\Omega)$ and
\begin{equation}\label{eq:W2p-est-A}
\|u\|_{W^{2,p}(\Omega)}
\le C\left(\|g\|_{L^{p}(\Omega)}+\|u\|_{L^{2}(\Omega)}\right).
\end{equation}
The constant $C$ depends on $p$, $\Omega$, $k$ and $\|\mu\|_{W^{1,p_0}(\Omega)}$. 
\end{proposition}
\begin{proof} 
All cases are immediate consequences of the resolvent estimate \cite[(3.3)]{DPF}
combined with the Beltrami conductivity correspondence.
\end{proof}

\subsection{Parabolic maximal regularity   and first-order smoothing}
\label{subsec:MR-W-1q}

Fix $q\in(2,\infty)$ and set \[X_q\coloneqq W^{-1,q}(\Omega)=(W^{1,q'}_0(\Omega))'.\]
Throughout we use the equivalent norm
\begin{equation}\label{eq:Xq-norm}
\|\Phi\|_{X_q}\coloneqq
\inf\left\{\|F\|_{L^q(\Omega)}+\|G\|_{L^q(\Omega)}:\ \Phi={\mathrm{div}} F+G\ \text{in }\mathcal D'(\Omega)\right\}.
\end{equation}
Equivalence with the standard dual norm follows from the boundedness of $\Omega$ and
$\partial\Omega\in \mathcal C^2$; see e.g. \cite[Theorem~III.2.43]{Boyer-Fabrie-Book}.
Define the operator
$
A_{\mu,q}:D(A_{\mu,q})\subset X_q\to X_q
$
by
\begin{equation}\label{eq:def-Amuq}
D(A_{\mu,q})\coloneqq W^{1,q}_0(\Omega)\subset X_q,
\qquad
\langle A_{\mu,q}u,\varphi\rangle_{X_q,W^{1,q'}_0}
\coloneqq \int_\Omega \sigma_\mu \nabla u\cdot \nabla\varphi\,\d x
\end{equation}
for all $u\in W^{1,q}_0(\Omega)$ and $\varphi\in W^{1,q'}_0(\Omega)$.
Equivalently, $A_{\mu,q}u=-\mathrm{div}(\sigma_\mu\nabla u)$ in $\mathcal D'(\Omega)$, and the right-hand side
belongs to $W^{-1,q}(\Omega)$ because $\sigma_\mu\nabla u\in L^q(\Omega;\mathbb R^2)$.
We now recall a classical maximal regularity result for the abstract Cauchy problem
\begin{equation}\label{eq:ACP-Xq}
\pt z + A_{\mu,q}z = \Phi,\qquad z(0)=z_0\in X_q, \qquad \Phi\in L^r_{\mathrm{loc}}(\mathbb R^+;X_q).
\end{equation}
A standard reference for the general theory of maximal regularity is the treatise \cite{KunstmannWeis2004}. We refer instead to  a precise result on distribution spaces, see \cite{HR}.
\begin{proposition}\cite[Theorem 5.4]{HR}
\label{prop:MR-W-1q}
Assume \eqref{e:ellipticitymu} and let $r\in (1,\infty)$. Then $A_{\mu,q}$ has maximal parabolic $L^r$ regularity on $X_q$.
That is,  for every $T>0$ and every $\Phi\in L^r(0,T;X_q)$ there exists a unique
\[
z\in W^{1,r}(0,T;X_q)\cap L^r(0,T;D(A_{\mu,q}))
\]
satisfying \eqref{eq:ACP-Xq} a.e.\ on $(0,T)$, together with the estimate
\begin{equation}\label{eq:MR-est}
\|z\|_{W^{1,r}(0,T;X_q)}+\|z\|_{L^r(0,T;D(A_{\mu,q}))}
\le C_{\mathrm{MR}}\left(\|z_0\|_{X_q}+\|\Phi\|_{L^r(0,T;X_q)}\right),
\end{equation}where $C_{\mathrm{MR}}=C_{\mathrm{MR}}(r,q,k,\Omega,T)$.
\end{proposition}
 
The next proposition packages the \emph{Beltrami-specific identification} of $D(A_{\mu,q})$ with
$W^{1,q}_0(\Omega)$ and the resulting $L^\infty$ smoothing.  

\begin{proposition} 
\label{prop:first-order-smoothing-MR}
Assume \eqref{e:ellipticitymu} holds and   fix $2<q<q_*(k)$.
Let $r\in(1,\infty)$, $T>0$, and let $z$ be the unique solution to \eqref{eq:ACP-Xq} given by
Proposition~\ref{prop:MR-W-1q}. Then for every $0<T_1<T_2$ 
\begin{equation}\label{eq:MR-W1q-est}
\|z\|_{L^r(T_1,T_2;W^{1,q})}
\le C\,
\left(\|A_{\mu,q}z\|_{L^r(T_1,T_2;X_q)}+\|z\|_{L^r(T_1,T_2;L^2)}\right),
\end{equation}
\end{proposition}
\begin{remark} \label{rem:29} Notice that, for the solution $z$ to \eqref{eq:ACP-Xq},
\[
\begin{split}
 \|A_{\mu,q}z\|_{L^r(T_1,T_2;X_q)}&=\|\Phi-\pt z\|_{L^r(T_1,T_2;X_q)} \leq \|z\|_{W^{1,r}(0,T;X_q)}+ \| \Phi\|_{L^r(T_1,T_2;X_q)},  \\  \|z\|_{L^r(T_1,T_2;L^2)} &\leq C \|z\|_{L^r(T_1,T_2;D(A_{\mu,q}))} =C \|z\|_{L^r(T_1,T_2;W_0^{1,q}(\Omega))}.
 \end{split}
\]
Inserting the last display into  \eqref{eq:MR-W1q-est}, and  appealing to \eqref{eq:MR-est} together with Sobolev embedding $W^{1,q}_0(\Omega)\hookrightarrow L^\infty(\Omega)$ leads to the \emph{a priori estimate} for the solution of \eqref{eq:ACP-Xq}
\begin{equation}\label{eq:MR-Linfty-est}
\|z\|_{L^r(T_1,T_2;L^\infty(\Omega))}
\le C \|z\|_{L^r(T_1,T_2;W^{1,q}_0(\Omega))} \leq C \left(\|z_0\|_{X_q}+\|\Phi\|_{L^r(0,T;X_q)}\right).
\end{equation}
\end{remark}
\begin{proof}[Proof of Proposition \ref{prop:first-order-smoothing-MR}]
By  Proposition~\ref{prop:MR-W-1q}, we have
$z\in W^{1,r}(0,T_2;X_q)\cap  L^r(0,T_2;D(A_{\mu,q}))$, hence
$A_{\mu,q}z\in L^r(0,T_2;X_q)$ and in particular for a.e.\ $t\in [0,T_2]$,
\[
A_{\mu,q}z(t)\in X_q,\qquad \text{and } \ A_{\mu,q}z(t)={\mathrm{div}} F(t)+G(t)
\]
for some $F(t),G(t)\in L^q$ with $\|F(t)\|_{L^q}+\|G(t)\|_{L^q}$ arbitrarily close to
$\|A_{\mu,q}z(t)\|_{X_q}$ by the definition \eqref{eq:Xq-norm}.
Applying Proposition~\ref{prop:beltrami-W1q} pointwise in $t$ with $u=z(t)$ gives
\[
\|\nabla z(t)\|_{L^q}
\le C(q,k,\Omega)\,\mathcal C_{q,\mu}\,
\left(\|A_{\mu,q}z(t)\|_{X_q}+\|z(t)\|_{V^0}\right)
\quad\text{for a.e.\ }t,
\]
which implies \eqref{eq:MR-W1q-est} upon integration over $(T_1,T_2)$.
\end{proof}

\section{Main results}\label{sec:mainres}

This section presents the main dynamical results for the 
Coleman--Gurtin-Beltrami model \eqref{eq:CG-QC}.  The results we obtain do parallel, to a wide extent, the 
dynamical scheme of \cites{CGDP2010,AMN} for the three-dimensional homogeneous setting. On the other hand, the regularization mechanisms must differ, as explained in \S\ref{Ss15},  in order to account for the more delicate Beltrami-specific first and second order
elliptic estimates    
\subsection{Standing assumptions and basic dissipative structure}\label{subsec:main-assumptions}

Throughout, $\Omega\subset\R^2$ is a bounded domain with $\partial\Omega\in C^2$.
We assume the uniform ellipticity condition \eqref{e:ellipticitymu},  on the Beltrami coefficient $\mu$,
the kernel hypotheses \textbf{(K1)}--\textbf{(K2)} with normalization \eqref{e:kappanorm}
the forcing assumption \textbf{(F)}, and the nonlinearity assumptions \textbf{(P1)}--\textbf{(P2)}. Below, the  
  \emph{structural constants} $C>1,0<c<1$, 
   and \emph{structural positive increasing functions} $\mathcal I$
 are allowed to vary at each instance, and depend only on the physical parameters involved in the standing assumptions recalled above or otherwise specified when necessary. Throughout,   $B_{\mathcal X}(R)$ stands for the ball of radius $R$ and center at the origin in the Banach norm $\mathcal X$.

The   relevant phase spaces in this section are ${\mathcal H^0}, {\mathcal H^1} $ and $\mathcal V$,   introduced in \eqref{eq:phase-spaces}. We begin by stating the following two background results, whose proof largely mirrors the Euclidean case.

\begin{proposition}
	
 \label{prop:main-H0}
  Problem \eqref{eq:CG-QC} generates a strongly continuous semigroup
$S(t):\mathcal H^{0}\to\mathcal H^{0}$ admitting  an absorbing ball
$\mathbb B_{0}\subset\mathcal H^{0}$.
Moreover, for each $R>0$    \begin{equation}
	\label{e:H0dis}
	\sup_{\xi \in B_{\mathcal H^0}(R)} \left[
\|S(t)\xi\|_{\mathcal H^0}^2+\int_t^{t+1}\|u_\xi(\tau)\|_{{V^1}}^2\,\d\tau \right]\le \mathcal I(R) \e^{-ct}+C, \qquad t>0
\end{equation}

\end{proposition}

\begin{proposition} \label{prop:main-H1-link}
The $\mathcal H^{0}$-semigroup $S(t)$ generated by Problem \eqref{eq:CG-QC} restricts to a strongly continuous semigroup on $\mathcal H^{1}$
admitting an absorbing ball $\mathbb B_{1}\subset\mathcal H^{1}$.
In addition, with reference to the absorbing ball $\mathbb B_{0}\subset\mathcal H^{0}$,
\begin{equation}\label{eq:link} \sup_{z\in \mathbb B_0}
\dist_{\mathcal H^0}\left(S(t)z,\mathbb{B}_{1}\right)\le C\e^{-c t}, \qquad t>0.
\end{equation}
\end{proposition}

\begin{remark} \label{rem:tool-link} No Beltrami regularity beyond \eqref{e:ellipticitymu} is needed here, as the argument relies only on
uniform ellipticity, the memory dissipation mechanisms, and the dissipativity of $\phi$.
\end{remark}

\subsection{Time-averaged and $\mathcal V$-instantaneous smoothing}\label{subsec:main-smoothing}
The strategy we follow aims at partial instantaneous regularization of the solution with $\mathcal H^1$ datum.
 More precisely, the $u$-component of the solution to \eqref{eq:CG-QC} on $\mathcal H^1$ enters ${V^2}$ as soon as evolution begins. 
 
 \begin{theorem} \label{thm-inst}
For  each $R>0$ and $z\in   B_{\mathcal H^1}(R)$, there holds
\[
\|S(t)z\|_{\mathcal V}\le \mathcal I(R)\left(1+\frac{1}{\sqrt{t}}\right), \qquad t>0.
\]

\end{theorem}
We extract the partial instantaneous one derivative gain mentioned in the abstract, conditional to additional regularity on $\mu$, which is immediately obtained by combination of Theorem  \ref{thm-inst} with Proposition \ref{prop:W2p-A} and the definition of $\mathcal V$.
 \begin{corollary} \label{cor-inst}Suppose \eqref{e:regmu} holds with $p_0=2$. For  each $1<p<2$, $R>0$ and $z\in   B_{\mathcal H^1}(R)$, denoting  by $u(t)$ the first component of the solution $S(t)z$, we have
\[
\|u(t)\|_{W^{2,p}(\Omega)}\le \mathcal I(R,p)\left(1+\frac{1}{\sqrt{t}}\right), \qquad t>0.
\]
 \end{corollary}
\begin{remark}
\label{rem:key}
The keystone in the proof of Theorem \ref{thm-inst}  is the realization of an $L^\infty$-type control of the nonlinear term, which, in the Euclidean Laplacian setting of \cite{CGDP2010}, is achieved via Agmon's inequality, which is supported by their space ${V^2}=H^2(\Omega) \cap H^{1}_0(\Omega)$. On the other hand, in the present setting ${V^2}=D(A_\mu)$ is an abstract graph space for a rough
divergence-form operator, and it does not provide Agmon-type pointwise control. The $L^\infty $ control here is obtained instead via maximal $L^r, 1<r<\infty$ parabolic regularity, which is made possible by $W^{1,q}(\Omega)$  Beltrami resolvent theory, cf.\ Proposition \ref{prop:first-order-smoothing-MR}.
\end{remark} 
The ideas described in Remark  \ref{rem:key} are formalized in the result below, which is one of the main novelties of this article.   \begin{theorem}
 \label{thm:main-Linfty-smoothing}
With reference to Proposition \ref{prop:main-H1-link}, let $z_0\in \mathcal H^1$
 and $u=u(t)$ stand for the first component of $S(t) z_0$, $t\geq 0$. Then, for each $1<r<\infty$,\[
\sup_{t_0\geq 0} 
\|u\|_{L^r(t_0,t_0+1;L^\infty(\Omega))}\le \mathcal I\left(\|z_0\|_{\mathcal H^{1}},r\right).
\]
\end{theorem}

 \subsection{Regular   attractors}\label{subsec:main-expattr} Finally, relying upon Theorems \ref{thm-inst} and \ref{thm:main-Linfty-smoothing}, we study the asymptotic dynamics of Problem \eqref{eq:CG-QC}.
The results that follow are borne out of the Beltrami second-order regularity theory for operators with rough coefficients, and for this reason, we will assume  that $\mu \in W^{1,2}(\Omega)$, that is \eqref{e:regmu} holds with $p_0=2$, so that
\[
{V^2} = \mathcal D(A_\mu) \hookrightarrow W^{2,p}(\Omega), \qquad 1<p<2
\]
by Proposition \ref{prop:W2p-A}.
 For this reason, from now till the end of this section,  $\|\mu\|_{W^{1,2}(\Omega)}$ is considered an admissible structural parameter.
Associated with the time-independent forcing $f$ is the vector, cf.\ \cite[(4.1)]{CGDP2010}
\begin{equation}\label{eq:main-zf}
u_f\coloneqq \frac12 A_\mu^{-1}f,\qquad \eta_f(s)\coloneqq s\,u_f,\qquad z_f\coloneqq (u_f,\eta_f)
\end{equation}
which is the steady state of \eqref{eq:CG-QC} when $\phi=0$.

\begin{theorem} \label{thm:main-exp-attractor-V}
Assume  \eqref{e:regmu} holds with $p_0=2$. Then, the semigroup $S(t)$ of Proposition \ref{prop:main-H0} admits a strongly continuous compression to the phase space $\mathcal V$ possessing a regular exponential attractor. Namely, there exists a set
$\mathfrak E\subset\mathcal V$  with the following properties:
\begin{itemize}
\item[\textup{(i)}] $\mathfrak E$ is compact and has finite fractal dimension in $\mathcal V$;
\item[\textup{(ii)}] $\mathfrak E -z_f$ is bounded  and has finite fractal dimension in the spaces  $\mathcal H^3$ and  $ W^{2,p}(\Omega)\times  L^{2}_{h}(\R^{+};W^{2,p}(\Omega))$ for all $1<p<2$;

\item[\textup{(iii)}]  $\mathfrak E$ is positively invariant, i.e.\ $S(t)\mathfrak E\subset \mathfrak E$ for all $t\ge0$;
\item[\textup{(iv)}] $\mathfrak E$ exponentially attracts bounded subsets of $\mathcal V$, namely\[
\dist_{\mathcal V}\left(S(t)B_{\mathcal V}(R),\mathfrak E\right)\le \mathcal I(R)\e^{-c t},
\qquad \forall t\ge0;
\]
\item[\textup{(v)}] $\mathfrak E$ exponentially attracts bounded subsets of $\mathcal H^1$, in the $\mathcal H^1$-metric and in the stronger $\mathcal V$-metric once  detached from initial times, namely
\[
\dist_{\mathcal H^1}\left(S(t)B_{\mathcal H^1}(R),\mathfrak E\right) + \sqrt{t}\dist_{\mathcal V}\left(S(t)B_{\mathcal H^1}(R),\mathfrak E\right)\le\mathcal I(R)\e^{-c t},
\quad \forall t\ge0;
\]
\item[\textup{(vi)}] $\mathfrak E$ exponentially attracts bounded subsets of $\mathcal H^0$, in the $\mathcal H^0$-metric, namely
\[
\dist_{\mathcal H^0}\left(S(t)B_{\mathcal H^0}(R),\mathfrak E\right) \leq  \mathcal I(R)\e^{-c t},
\qquad \forall t\ge0.
\]
\end{itemize}
\end{theorem}
The next corollary concerns the existence of regular global attractors for the semigroup on both phase spaces $\mathcal H^0, \mathcal H^1.$
\begin{corollary}\label{cor:main-global-attractors}
Assume \eqref{e:regmu} holds with $p_0=2$. With reference to \textup{Theorem \ref{thm:main-exp-attractor-V}}, there exists a compact subset  $\mathfrak A\subset \mathfrak E$ with the properties that follow.
\begin{itemize}
\item[\textup{(i)}] $\mathfrak A$ is fully invariant under $S(t)$, that is  $S(t)\mathfrak A=\mathfrak A$, and the restriction of $S(t)$ on $\mathfrak A$ is a group of operators;
\item[\textup{(ii)}] $\mathfrak A$ attracts bounded sets in $\mathcal H^0$, $\mathcal H^1$ in  the following sense: for each
for every $R>0$,
\[
\lim_{t\to \infty}
\dist_{\mathcal H^0}\left(S(t)B_{\mathcal H^0}(R),\mathfrak A\right)= 0, \qquad 
\lim_{t\to \infty}
\dist_{\mathcal V}\left(S(t)B_{\mathcal H^1}(R),\mathfrak A\right)= 0.
\]
\end{itemize}
\end{corollary}

\begin{remark} \label{rem:beltrami-tools-main}
During the proof of Theorem \ref{thm:main-exp-attractor-V}, 
nonlinear differences must be bounded at the ${V^2}$ level. This is precisely where the second-order Beltrami theory of Proposition~\ref{prop:W2p-A} is used.
 \end{remark}

\section{Dissipativity in $\mathcal H^{j}$, $j=0,1$}\label{sec:dissipativity-Hj}

This section summarizes the proofs of Propositions \ref{prop:main-H0} and \ref{prop:main-H1-link}.     We restrict ourselves to  indicating the relevant differences with the previous treatment of the  constant coefficient case, whose main references are \cites{CGDP2010,AMN}.

\subsection{Proof of Proposition \ref{prop:main-H0}} First, once the dissipative estimate \eqref{e:H0dis} is   obtained, well-posedness follows by a standard Galerkin scheme which we omit.
To prove \eqref{e:H0dis}, preliminarily observe that assumptions \textbf{(P1)}-\textbf{(P2)} entail the coercivity
\begin{equation}\label{eq:phi-coercivity}
\phi(r)\,r \ge -(\underline\lambda(k,\Omega)-\delta)\,r^2 - C_\delta,
\qquad \forall r\in\R,
\end{equation}
where $\delta>0$ will be chosen later and $\underline\lambda(k,\Omega)$ is defined in \eqref{eq:uniform-coercivity}.
Write the system in the form \eqref{eq:CG-QCext} with $r=0$ and forcing
$g(t)=f-\phi(u(t))$. Applying Remark~\ref{cor:CG-energy} with $r=0$ yields, for $\nu\in(0,\nu_0)$,
\begin{equation}\label{eq:H0-damped-start}
\frac{\d}{\d t}\mathcal E_{0,\nu}(u(t),\eta^t)+\nu^2\mathcal E_{0,\nu}(u(t),\eta^t) + (1-\nu)\|u(t)\|^2_{{V^1}}
\le 2|\langle f-\phi(u(t)),u(t)\rangle_{V^0}|.
\end{equation}
Coercivity \eqref{eq:uniform-coercivity},  standard H\"older and Poincar\'e inequality, and a suitable choice of $\delta$ depending only on $\nu_0$ lead us to the differential inequality 
\[
\frac{\d}{\d t}\mathcal  E_{0,\nu}(u(t),\eta^t)+\nu^2\mathcal E_{0,\nu}(u(t),\eta^t)+\nu  \|u(t)\|^2_{{V^1}} \le C(1+\|f\|_{V^0}^2),
\]
which leads to \eqref{e:H0dis} by Gronwall's lemma.

\subsection{Proof of dissipativity in Proposition \ref{prop:main-H1-link}}
First, we show that Problem \eqref{eq:CG-QC} also generates a  
strongly continuous semigroup
\[
S(t):\mathcal H^{1}\to\mathcal H^{1},\qquad S(t)(u_0,\eta_0)=(u(t),\eta^{t})
\] 
satisfying the absorbing estimate and dissipation integral
\begin{equation} \label{e:abs1}
	\sup_{\xi=(u,\eta) \in B_{R}({\mathcal H^1})} \left[ \left\| S(t) \xi\right\|_{{\mathcal H^1}}^2 + \int_{t}^{t+1} \|u_\xi (\tau)\|_{  {V^2}}^2 \, \d \tau \right] \leq \mathcal I (R)\e^{-ct} +C, \qquad t\geq 0.
\end{equation} 
where $S(t)\xi=(u_\xi(t),\eta_\xi^t)$.
In particular, the ball  $\mathbb B_1\coloneqq B_{2C}({\mathcal H^1})$, with reference to the constant $C$ appearing in  \eqref{e:abs1}, is an absorbing set for the semigroup $S(\cdot)$ on $\mathcal H^1$.

Let $\xi\in B_R(\mathcal H^1)$ and let $(u,\eta)=S(\cdot)\xi$. By Proposition \ref{prop:main-H0},
the trajectory enters a bounded absorbing set of $\mathcal H^0$. In particular, there exists $t_0=\mathcal I(R)$ such that
\begin{equation}\label{eq:H0-bdd-used-new}
\sup_{t\ge t_0}\|u(t)\|_{V^0}
+
\sup_{t\ge t_0}\int_t^{t+1}\|u(\tau)\|_{{V^1}}^2\,\d\tau \le C. 
\end{equation}
Now apply Remark~\ref{cor:CG-energy} with $r=1$ to \eqref{eq:CG-QCext} and forcing
$g(t)=f-\phi(u(t))$. This yields, for $\nu\in(0,\nu_0)$,
\begin{equation}\label{eq:H1-damped-start}
\frac{\d}{\d t}\mathcal E_{1,\nu}(u(t),\eta^t)+\nu^2\mathcal E_{1,\nu}(u(t),\eta^t) + (1-\nu) \|u(t)\|_{{V^2}}^2
\le 2 \langle f-\phi(u(t)),u(t)\rangle_{{V^1}},
\end{equation}
having implemented the shorthand notation $\mathcal E(t)\coloneqq \mathcal E_{1,\nu}(u(t),\eta^t).$
First, 
\[ 2
\langle f ,u\rangle_{{V^1}}
\le 2\|f\|_{V^0}\|u\|_{{V^2}} \leq \nu \|u\|_{{V^2}}^2 + C \|f\|_{V^0}^2.
\]
Further, integrating by parts and using the lower bound (P1) followed by ellipticity, 
\begin{equation}\label{eq:phi-H1-split}
-\langle \phi(u),u\rangle_{{V^1}} = 
\int_\Omega \phi'(u)\,\sigma_\mu\nabla u\cdot\nabla u\,\d x \leq C\|u\|_{{V^1}}^2.
\end{equation} Combining \eqref{eq:H1-damped-start} with the last two displays leads to the differential inequality
\begin{equation}\label{eq:H1-damped-middle}
\frac{\d}{\d t}\mathcal E(t)+\nu^2\mathcal E(t)+ (1-2\nu) \|u(t)\|_{{V^2}}^2
\le  C   (1+ \|u(t)\|_{{V^1}}^2)
\end{equation} 
A first Gronwall integration of the crude version of \eqref{eq:H1-damped-middle}, namely $\mathcal E'\leq C \mathcal E +C,$ yields $ \mathcal E(t_0) \leq  \mathcal I(t_0,R)$. Thus ultimately $\mathcal E(t_0) \leq  \mathcal I(R)$.
By variation of constants in \eqref{eq:H1-damped-middle}  on $(t_0,t)$,
\begin{equation}\label{eq:VOC-E}
\begin{split}
\mathcal E(t)&\le \e^{-\nu^2(t-t_0)}\mathcal E(t_0)+C\left(1-\e^{-\nu^2(t-t_0)}\right)+\int_{t_0}^t e^{-\nu^2(t-s)}\|u(s)\|_{{V^1}}^2\,\d s 
\\ &\leq \mathcal I(R) \e^{-\nu^2(t-t_0)} + C
\end{split}
\end{equation}
where the integral term has been estimated via the translation-boundedness in \eqref{eq:H0-bdd-used-new} and exponential decay.
Recalling the norm equivalence \eqref{eq:E1-equiv-cor}, \eqref{eq:VOC-E} leads to the claimed absorption property.

\subsection{Proof of (\ref{eq:link}) of Proposition \ref{prop:main-H1-link}} 
The crux here is to prove that there exist  $R_1>0$  depending only on the structural parameters such that for every $z\in \mathbb B_0$ and every $t\ge 0$,
\begin{equation}\label{eq:exp-attraction}
\dist_{\mathcal H^0} (S(t)z,\; B_{\mathcal H^1}(R_1)) \le C\e^{-c t}.
\end{equation}
In particular, the ball $B_{\mathcal H^1}(R_1)$ exponentially attracts $\mathbb B_0$ in the $\mathcal H^0$--metric. Redefining $\mathbb  B_1$ to be larger than $B_{\mathcal H^1}(R_1)$ if necessary, we obtain the existence of a $\mathcal H^1 $-absorbing ball that also attracts bounded sets of $\mathcal H^0$ exponentially, as claimed.

Begin the proof proper by
fixing $z=(u_0,\eta_0)\in\mathbb B_0$, and writing $S(t)z=(u(t),\eta^t)$.
Relying on (P1), choose a structural $\ell\ge 0$ so that the shifted nonlinearity $\phi_0(r)=\phi(r)+\ell r$ is monotone, that is
$\phi_0'(r)\ge 0$ for all $r\in\mathbb R$.
Decompose
\[
u=v+w,\qquad \eta=\xi+\zeta,
\]
where $(v,\xi)$ and $(w,\zeta)$ solve the coupled systems
\begin{equation}\label{eq:L-system}
\left\{
\begin{aligned}
&\partial_t v + A_\mu v + \int_0^\infty h(s)A_\mu\xi(s)\, \d s   =\phi_0(w)-\phi_0(u) \\
&\partial_t \xi = T\xi + v,\\
&(v(0),\xi^0)=(u_0,\eta_0),
\end{aligned}
\right.
\end{equation}
and
\begin{equation}\label{eq:K-system}
\left\{
\begin{aligned}
&\partial_t w + A_\mu w + \int_0^\infty h(s)A_\mu\zeta(s) \,\d s 
= f - (\phi_0(w)-\ell u),\\
&\partial_t \zeta = T\zeta + w,\\
&(w(0),\zeta^0)=(0,0).
\end{aligned}
\right.
\end{equation}
Adding \eqref{eq:L-system} and \eqref{eq:K-system} yields the original system, hence
\begin{equation}\label{eq:decomp}
S(t)z = L(t)z + K(t)z,\qquad L(t)z\coloneqq (v(t),\xi^t),\quad K(t)z\coloneqq(w(t),\zeta^t).
\end{equation}
Let us handle \eqref{eq:L-system} first. Apply Remark~\ref{cor:CG-energy} with $r=0$, viewing
$
g_L(t)\coloneqq -\left(\phi_0(u(t))-\phi_0(w(t))\right)
$
as the forcing in \eqref{eq:CG-QCext}. For $\nu\in(0,\nu_0)$ we obtain the damped inequality
\begin{equation}\label{eq:L-damped-start}
\frac{\d}{\d t}\mathcal E_{0,\nu}\left(v(t),\xi^t\right)+\nu^2\mathcal E_{0,\nu}\left(v(t),\xi^t\right)
\le 2\langle g_L(t),v(t)\rangle_{V^0} .
\end{equation}
Since $\phi_0$ is monotone,
  $\langle g_L(t),v(t)\rangle_{V^0}\le 0$ and the right-hand side of
\eqref{eq:L-damped-start} is negative definite. Therefore,
\[
\mathcal E_{0,\nu}\left(v(t),\xi^t\right)\le \e^{-\nu^2 t}\mathcal E_{0,\nu}(u_0,\eta_0).
\]
Using the norm equivalence \eqref{eq:E1-equiv-cor} at level $r=0$ we conclude that, for some $C\ge 1$,
\begin{equation}\label{eq:L-exp}
\|L(t)z\|_{\mathcal H^0}
=\|(v(t),\xi^t)\|_{\mathcal H^0}
\le C\e^{-\frac{\nu^2}{2}t}\|z\|_{\mathcal H^0}\leq  C\e^{-\frac{\nu^2}{2}t},
\end{equation}
as $z\in\mathbb B_0$
Thus there exist $C\ge 1,c>0$  such that
\begin{equation}\label{eq:L-exp-B0}
\|L(t)z\|_{\mathcal H^0}\le C\e^{-c t}\qquad \forall t\ge 0,\ \forall z\in\mathbb B_0.
\end{equation}
Turn to \eqref{eq:K-system} and
apply Remark~\ref{cor:CG-energy} with $r=1$, so that\begin{equation}\label{eq:K-energy-start}
\frac{\d}{
\d t}\mathcal E_{1,\nu}\left(w(t),\zeta^t\right)+\nu^2\mathcal E_{1,\nu}\left(w(t),\zeta^t\right) + (1-\nu)  \|w\|_{{V^2}}^2 
\le 2 \langle f+\ell u(t)-\phi_0(w(t)),\,w(t)\rangle_{{V^1}} .
\end{equation}
We estimate
\[
2|\langle f+\ell u(t),w\rangle_{{V^1}}|
\le \nu\|w\|_{{V^2}}^2 + C \|f\|_{V^0}^2+ C\|u(t)\|_{V^0}^2 \leq \nu\|w\|_{{V^2}}^2+C,
\]
where the last inequality is due to the absorbing estimate in $\mathcal H^0.$
On the other hand, since $\phi_0'(r)\ge 0$ for all $r$, integration by parts yields
\[
-\langle \phi_0(w),w\rangle_{{V^1}}
=\int_\Omega \phi_0'(w)\,\sigma_\mu\nabla w\cdot\nabla w\,\d x \ge\ 0,
\]
so the contribution of $-\phi_0(w)$ in \eqref{eq:K-energy-start} is dissipative and may be dropped.
Proceeding exactly as in Remark~\ref{cor:CG-energy}, in particular absorbing the $\nu\|w\|_{{V^2}}^2$ term into the left-hand side,
we obtain
\begin{equation}\label{eq:K-damped}
\frac{\d}{\d t}\mathcal E_{1,\nu}\left(w(t),\zeta^t\right)+\nu^2\mathcal E_{1,\nu}\left(w(t),\zeta^t\right) + (1-2\nu)    \|w\|_{{V^2}}^2 
\le C. 
\end{equation}
Using $\mathcal E_{1,\nu}(w(0),\zeta^0)=0$,  integrating \eqref{eq:K-damped} and taking into account the norm equivalence \eqref{eq:E1-equiv-cor} yields
the existence of $R_1>0$ such that
\begin{equation}\label{eq:K-H1-bdd}
\|K(t)z\|_{\mathcal H^1}\le R_1
\qquad \forall\,t\ge 0,\ \forall\,z\in\mathbb B_0,
\end{equation}
and $R_1$ depends only on the structural parameters.
Invoking \eqref{eq:L-exp-B0}, we have shown 
\[
\dist_{\mathcal H^0}\!\left(S(t)z,\;B_{\mathcal H^1}(R_1)\right)\le C\e^{-c t},
\]
which is \eqref{eq:exp-attraction}.

 \section{Proof of Theorem \ref{thm:main-Linfty-smoothing}}
 \label{s5}
 
 Let
$r\in (1,\infty)$, and  for each $z_0=(u_0,\eta_0)\in {\mathcal H^1}$, denote  the corresponding
trajectory by $S(t)z_0\coloneqq (u(t),\eta^t)$. Recall that we have to prove
\begin{equation}\label{eq:MR-Linfty-on-B1}
\|u\|_{L^r(t_0,t_0+1;L^\infty(\Omega))}\le \mathcal I(\|z_0\|_{{V^1}},r), \qquad t_0\ge 0.
\end{equation}
For simplicity, from now on, write $R=\|z_0\|_{\mathcal H^1}$.
For $t\ge0$, rewrite the $u$-equation from \eqref{eq:CG-QC} as
\begin{equation}\label{eq:u-MR-form}
\pt u(t)+A_\mu u(t)=\Phi(t) \coloneqq f-\phi(u(t))+\int_0^\infty h(s)A_\mu\eta^t(s)\,\d s \qquad \text{in }\mathcal D'(\Omega),
\end{equation}
Fix $t_0\ge0$ and consider \eqref{eq:u-MR-form} on the time interval $(t_0,t_0+1)$. Fix $q=q(k)\coloneqq \frac{4(k+1)}{3k+1} $, so that $\frac1q$ the midpoint of $(\frac{1}{q_*(k)}, \frac{1}{2})$, which obviously complies with assumptions of Proposition \ref{prop:MR-W-1q}. Recall $X_q=W^{-1,q}(\Omega)=(W_0^{1,q'}(\Omega))'$ and the notation  $A_{\mu,q}$ for the
realization \eqref{eq:def-Amuq} on $X_q$.
As $q>2$,  the two-dimensional Sobolev embedding
$W_0^{1,q'}(\Omega)\hookrightarrow L^2(\Omega)=V^0$ reads 
\begin{equation}\label{eq:L2-into-Xq}
\|g\|_{X_q}\le C\,\|g\|_{V^0}\qquad \forall g\in V^0
\end{equation}
The dissipative estimate  \eqref{e:H0dis} and  \eqref{eq:L2-into-Xq} combine to tell us in particular that
\begin{equation}\label{eq:uL2-unif} \|u(t_0)\|_{X_q} \leq C
\sup_{t\ge0}\|u(t)\|_{V^0}\le \mathcal I(R).
\end{equation}
We estimate the three contributions to $\Phi(t)$ in $V^0$ as well, then use \eqref{eq:L2-into-Xq} again.
Obviously  $\|f\|_{X_q}\le C\|f\|_{V^0}$.
Further, via  Cauchy-Schwarz and memory kernel assumptions,
\[ 
\begin{split}
\left\|\int_0^\infty h(s) A_\mu\eta^t(s)\,\d s\right\|_{V^0}
&\le \int_0^\infty h(s)\|A_\mu\eta^t(s)\|_{V^0}\,\d s
\le \kappa_0^{{\frac12}}\left(\int_0^\infty h(s)\|\eta^t(s)\|_{{V^2}}^2\,\d s\right)^{\frac12} \\ &
= \kappa_0^{\frac12}\,\|\eta^t\|_{\mathcal M^2}.
\end{split}
\]
Further, estimate \eqref{e:abs1}  reads
\[\sup_{t\ge0}\|\eta^t\|_{\mathcal M^2}\le \mathcal I(R).\] Hence, the memory term is bounded in
$L^\infty(t_0,t_0+1;V^0)$, thus in $L^r(t_0,t_0+1;X_q)$.
Moving to the nonlinear term, by \textbf{(P2)} and $\phi(0)=0$ there exists $C>0$ such that
$|\phi(s)|\le C(1+|s|^{m+1})$ for all $s\in\mathbb R$. Therefore  
\[
\|\phi(u(t))\|_{V^0}
\le C\left(|\Omega|^{1/2}+\|u(t)\|_{L^{2(m+1)}}^{m+1}\right) \leq C\left(1+ \|u(t) \|_{V^1}^{m+1}\right) \leq \mathcal I(R), \qquad t \geq t_0
\]
relying upon Sobolev embeddings $H_0^1(\Omega) \sim {V^1}\hookrightarrow L^{p}(\Omega)$ for each   $1<p<\infty$ and again on \eqref{e:abs1}.
Collecting, we obtain
\begin{equation}\label{eq:Phi-Xq-bound}
\|\Phi\|_{L^\infty(t_0,t_0+1;X_q)}\le  \mathcal I(R).
\end{equation}
The abstract equation \eqref{eq:u-MR-form} is precisely \eqref{eq:ACP-Xq} on $(t_0,t_0+1)$, namely
\[
\pt u + A_{\mu,q}u=\Phi \quad \text{in }X_q.
\]
Via Proposition~\ref{prop:first-order-smoothing-MR} and Remark \ref{rem:29}, maximal $L^r$ regularity for each $1<r<\infty$ yields in particular 
\[
\|u\|_{L^r(t_0,t_0+1;L^\infty(\Omega))}
\le \mathcal I(r)\left(\|u(t_0)\|_{X_q}+\|\Phi\|_{L^r(t_0,t_0+1;X_q)}\right)
\le \mathcal I(R,r),
\]
which is \eqref{eq:MR-Linfty-on-B1}.

\section{Instantaneous smoothing and continuous dependence}\label{sec:aux-Vp}
In this section, we leverage Theorem \ref{thm:main-Linfty-smoothing} to show that the semigroup  $S(t) $ generated by \eqref{eq:CG-QC} enjoys instantaneous regularization properties into the space $\mathcal V$. In the same stroke, we   obtain that $S(t) $ is  also a $\mathcal V$-semigroup, enjoying strong-type local continuous dependence provided $\mathcal V$ supports elliptic regularity. These dynamical results will be exploited in the construction of regular exponential attractors, which may thus be performed  at the $\mathcal V$-level.  The first rigorous statement is as follows.

\begin{proposition} \label{prop:V-smoothing}
Let $R>0$ and let $z\in B_{\mathcal H^1}(R)$. Then, for every $t>0$, one has $S(t)z\in\mathcal V$ and
\begin{equation}\label{eq:V-smoothing-est}
\|S(t)z\|_{\mathcal V}\le \mathcal I(R)\left(1+t^{-1/2}\right).
\end{equation}
If in addition $z\in \mathcal V$, then
\begin{equation}\label{eq:V-smoothing-est-V0}
\|S(t)z\|_{\mathcal V}\le \mathcal I(R)\,\|z\|_{\mathcal V}\e^{-t}+\mathcal I(R), \qquad t\ge 0.
\end{equation}
As a consequence, $S(t)$ admits a $({\mathcal H^1},\mathcal V)$-absorbing ball  $\mathbb B_{\mathcal V}\coloneqq B_{\mathcal V}(R_{\mathcal V})$; namely, for each $R>0$ there exists $t_{\mathcal V}=\mathcal I(R)>0$ such that
\[
S(t)B_{\mathcal H^1}(R)\subset \mathbb B_{\mathcal V},\qquad \forall t\ge t_{\mathcal V}.
\]

\end{proposition}
Up to the construction of the $\mathcal V$--absorbing set, the analysis relies only on the abstract Hilbert scale
$V^r=D(A_\mu^{r/2})$ and on maximal regularity on $W^{-1,q}$, hence the ellipticity bound \eqref{e:ellipticitymu} is sufficient. 
The continuous dependence estimate for the semigroup in the space  $\mathcal V$, however, requires a pointwise-in-time ${V^1}$  control on  the nonlinear
difference so as to comply with Lemma~\ref{lem:CGDP55-56-updated} \textup{(B)} below. 
This entails estimating products involving
$\|u_2(t)-u_1(t)\|_{\infty}$ at fixed times. 
Since ${V^2}=D(A_\mu)$ is only an abstract graph space for rough coefficients, such $L^\infty$ bounds are not available
from the Hilbert scale alone. For this reason, we are bound to invoke the Beltrami elliptic regularity
Proposition~\ref{prop:W2p-A}, which under $\mu\in W^{1,2}$  yields the
crucial embeddings ${V^2}\hookrightarrow W^{2,p}(\Omega)\hookrightarrow L^\infty(\Omega)$ in dimension $2$, for $1<p<2.$

\begin{proposition}[Continuous dependence in $\mathcal V$]\label{prop:contdep-V}
Assume, in addition to \eqref{e:ellipticitymu}, that \eqref{e:regmu} holds with $p_0= 2$.
Let $R>0$ and let $z_1,z_2\in B_{\mathcal V}(R)$. Denote
\[
S(t)z_i=(u_i(t),\eta_i^t),\qquad w(t)\coloneqq u_2(t)-u_1(t),\qquad \bar\eta^t\coloneqq \eta_2^t-\eta_1^t,
\qquad W(t)\coloneqq (w(t),\bar\eta^t).
\]
Then  
\begin{equation}\label{eq:contdep-V}
\|W(t)\|_{\mathcal V}\le \mathcal I(R)\|W(0)\|_{\mathcal V}\,\e^{\mathcal I(R) t}\qquad \forall t\ge 0.
\end{equation}
\end{proposition}
The remainder of the section is occupied by the proofs of Propositions \ref{prop:V-smoothing} and \ref{prop:contdep-V}. It is easier to first collect the steps common to both into an estimate for the auxiliary problem \eqref{eq:aux-CG} below. The proofs proper occupy the last two subsections.

\subsection{The auxiliary problem in $\mathcal H^{1}$}\label{subsec:aux-problem}

Given $z\in\mathcal H^{1}$, $f\in {V^0}$, and a forcing
$g\in L^{2}_{\mathrm{loc}}(\mathbb R^{+};{V^1})$,   consider the Cauchy problem
\begin{equation}\label{eq:aux-CG}
\begin{cases}
\pt z + A_{\mu}z + \displaystyle\int_{0}^{\infty}h(s)A_{\mu}\xi(s)\,\d s = f+g(t),\\[2mm]
\pt \xi = T\xi + z,\\[1mm]
(z(0),\xi^{0})=(z_{0},\xi_{0})\in\mathcal H^{1}.
\end{cases}
\end{equation}
As explained in Section~\ref{sec:functional-setting}, the operator $-A_{\mu}$ generates an analytic
contraction semigroup on ${V^0}$, while $T$ generates the right-translation semigroup on each
${\mathcal M^{r}}$ (Lemma~\ref{lem:translation-dissipation}). Hence \eqref{eq:aux-CG} is an abstract
evolution equation on $\mathcal H^{1}$ and is well-posed in the standard mild sense.


\begin{lemma} \label{lem:CGDP55-56-updated}
Let $(z,\xi)$ be a solution of \eqref{eq:aux-CG} on $[0,\infty)$ with initial datum
$Z_0=(z(0),\xi^0)\in\mathcal H^{1}$, and let $g\in L^2_{\loc}(\R^+;{V^1})$.

\smallskip
\noindent\textup{(A)} 
Assume that there exist constants $\beta,\gamma\ge0$ such that
\begin{equation}\label{eq:CGDP55-ass1-updated}
\sup_{t\ge0}\int_t^{t+1}\|(z(\tau),\xi^\tau)\|_{\mathcal V}^2\,\d\tau \le \beta,
\qquad 
\sup_{t\ge0}\int_t^{t+1}\|g(\tau)\|_{{V^1}}^2\,\d\tau \le \gamma .
\end{equation}
Then there exists a constant $D=D(\beta,\gamma,\|f\|_{V^0})>0$ such that
\begin{equation}\label{eq:CGDP55-concl-updated}
\|(z(t),\xi^t)\|_{\mathcal V}\le D\left(1+t^{-1/2}\right)
\qquad \forall\,t>0.
\end{equation}
If in addition $Z_0\in\mathcal V$, then
\begin{equation}\label{eq:CGDP55-concl-V0-updated}
\|(z(t),\xi^t)\|_{\mathcal V}\le D\|Z_0\|_{\mathcal V}\e^{-t}+D
\qquad \forall\,t\ge0.
\end{equation}

\smallskip
\noindent\textup{(B)}
Assume $f=0$ and $Z_0\in\mathcal V$, and that there exists $\upsilon \ge0$ such that
\begin{equation}\label{eq:CGDP56-ass-updated}
\|g(t)\|_{{V^1}}\le  \upsilon\|(z(t),\xi^t)\|_{\mathcal V}
\qquad\text{a.e. }t\ge0.
\end{equation}
Then there exist constants $D_j=D_j(\upsilon)>0, \, j=1,2$ such that
\begin{equation}\label{eq:CGDP56-concl-updated}
\|(z(t),\xi^t)\|_{\mathcal V}\le D_1\|Z_0\|_{\mathcal V}\,e^{D_2 t}
\qquad \forall\,t\ge0.
\end{equation}
\end{lemma}

\begin{proof}
The proof of \textup{(B)} is exactly the same as \cite[Lemma~5.6]{CGDP2010} with $A=-\Delta$ replaced by $A_\mu$,
since it uses only positivity, self-adjointness, compact resolvent of $A_\mu$, the Hilbert scale $V^r=D(A_\mu^{r/2})$,
dissipativity of $T$ on $\mathcal M^r$  from Lemma~\ref{lem:translation-dissipation}, and the memory tail identity
from Lemma~\ref{lem:memory-tail}.

\smallskip
For \textup{(A)}, we follow \cite[Lemma~5.5]{CGDP2010} and isolate the only point where the different hypothesis on $g$ is used.
As in \cite[Section~5]{CGDP2010}, one introduces a Lyapunov-type functional at level $r=1$,
obtained by taking the ${V^1}$-inner product of the first equation in \eqref{eq:aux-CG} with $\pt z$,
combining with the $\xi$-equation and the tail functional $\mathcal U_2[\xi]$ coming from Lemma~\ref{lem:memory-tail} with $r=1$,
and absorbing lower-order terms using \textbf{(K1)}--\textbf{(K2)}. This yields an absolutely continuous functional
$\mathcal Y:[0,\infty)\to\R$ satisfying the two properties

\smallskip
\noindent {(i) Equivalence.} There exist structural constants $c_0,C_0>0$ such that for all $t\ge0$,
\begin{equation}\label{eq:Y-equivalence} 
c_0\|(z(t),\xi^t)\|_{\mathcal V}^2 - C_0(1+\|f\|_{V^0}^2)\le \mathcal Y(t)
\le C_0\left(\|(z(t),\xi^t)\|_{\mathcal V}^2 + 1+\|f\|_{V^0}^2\right).
\end{equation}

\smallskip
\noindent {(ii) Differential inequality.} There exist structural constants $c_1,C_1>0$ such that for a.e.\ $t\ge0$,
\begin{equation}\label{eq:Y-differential} \
\frac{\d}{\d t}\mathcal Y(t) + c_1\|(z(t),\xi^t)\|_{\mathcal V}^2
\le C_1\left(1+\|f\|_{V^0}^2\right) + C_1\|g(t)\|_{{V^1}}^2.
\end{equation}
See \cite[(5.14)--(5.20)]{CGDP2010}. The derivation is purely functional and uses only the properties listed above,
so it carries over verbatim with $A_\mu$ in place of $-\Delta$.

At this stage, \cite[Lemma~5.5]{CGDP2010} estimates the forcing contribution via the pointwise growth assumption on $g$.
However, the subsequent argument uses \eqref{eq:Y-differential} only after integrating over unit intervals $(t,t+1)$ and
applying Cauchy--Schwarz and Young inequalities. Therefore the weaker translation-bounded assumption
$\sup_{t\ge0}\int_t^{t+1}\|g(\tau)\|_{{V^1}}^2\,\d\tau<\infty$ is sufficient.
Indeed, integrating \eqref{eq:Y-differential} over $(t,t+1)$ and using \eqref{eq:CGDP55-ass1-updated} gives
\[
\begin{split} &\quad
\mathcal Y(t+1)-\mathcal Y(t) + c_1\int_t^{t+1}\|(z(\tau),\xi^\tau)\|_{\mathcal V}^2\,\d\tau
\le C\left(1+\|f\|_{V^0}^2\right)+C\int_t^{t+1}\|g(\tau)\|_{{V^1}}^2\,\d\tau
\\ &\le C(\beta,\gamma,\|f\|_{V^0}),
\end{split}
\]
and from here the remainder of the proof proceeds exactly as in \cite[Lemma~5.5]{CGDP2010}.
\end{proof}

\subsection{Proof of Proposition~\ref{prop:V-smoothing}}
 
Let $Z(t)=(u(t),\eta^t)=S(t)z$. By \eqref{e:abs1}, there exists $\beta=\beta(R)\ge 0$ such that
\begin{equation}\label{eq:V-translation-bdd}
\sup_{t\ge 0}\int_t^{t+1}\|Z(\tau)\|_{\mathcal V}^2\,\d\tau
=\sup_{t\ge 0}\int_t^{t+1}\left(\|u(\tau)\|_{{V^2}}^2+\|\eta^\tau\|_{\mathcal M^2}^2\right)\,\d\tau
\le \beta.
\end{equation}
We apply Lemma~\ref{lem:CGDP55-56-updated} to the original system \eqref{eq:CG-QC}, viewed as the auxiliary problem
\eqref{eq:aux-CG} with forcing $g(t)=-\phi(u(t))$. 
Thus it remains to verify the corresponding growth assumption
on $g$.
Let $m\ge 1$ be the exponent in \textbf{(P2)}. Fix $t\ge 0$.
Since $q>2$ and $\Omega\subset\R^2$, the Sobolev embedding $W^{1,q}_0(\Omega)\hookrightarrow L^\infty(\Omega)$ holds.
Choose $r=2m$ in Theorem~\ref{thm:main-Linfty-smoothing}  to obtain the translation-bounded estimate
\begin{equation}\label{eq:u-Linfty-2m}
\sup_{t\ge 0}\ \|u\|_{L^{2m}(t,t+1;L^\infty(\Omega))} \le  \mathcal{I}(R).
\end{equation}
On the other hand, by ellipticity and the chain rule, almost everywhere in time
\[
\|\phi(u(\tau))\|_{{V^1}}^2
\lesssim_{k,\Omega} \|\nabla \phi(u(\tau))\|_{V^0}^2
=\|\phi'(u(\tau))\nabla u(\tau)\|_{V^0}^2
\le \|\phi'(u(\tau))\|_{\infty}^2\|\nabla u(\tau)\|_{V^0}^2.
\]
By \textbf{(P2)}, $|\phi'(s)|\le C(1+|s|^m)$, hence
\[
\|\phi'(u(\tau))\|_{\infty}^2 \le C\left(1+\|u(\tau)\|_{\infty}^{2m}\right).
\]
Moreover, Proposition~\ref{prop:main-H1-link} implies $\sup_{\tau\ge 0}\|u(\tau)\|_{{V^1}}\le \mathcal I(R)$, hence
$\sup_{\tau\ge 0}\|\nabla u(\tau)\|_{V^0}^2\le  \mathcal{I}(R)$ by \eqref{eq:form-bounds}.
Integrating over $(t,t+1)$ and using \eqref{eq:u-Linfty-2m} yields
\begin{equation}\label{eq:phi-H1-L2tb}
\int_t^{t+1}\|\phi(u(\tau))\|_{{V^1}}^2\,\d \tau
\le \mathcal{I}(R)\int_t^{t+1}\left(1+\|u(\tau)\|_{\infty}^{2m}\right)\,\d\tau
\le \mathcal{I}(R),
\qquad \forall t\ge 0.
\end{equation}
 Therefore Lemma~\ref{lem:CGDP55-56-updated}\textup{(A)} applies with $g=-\phi(u)$,
together with \eqref{eq:V-translation-bdd}, yielding \eqref{eq:V-smoothing-est}.
If $z\in\mathcal V$, the improved estimate \eqref{eq:V-smoothing-est-V0} follows from the second conclusion
\eqref{eq:CGDP55-concl-V0-updated} of Lemma~\ref{lem:CGDP55-56-updated} \textup{(A)}.
The existence of an $(\mathcal H^1,\mathcal V)$-absorbing ball as claimed
 is an immediate consequence of \eqref{eq:V-smoothing-est}, \eqref{eq:V-smoothing-est-V0} and of the semigroup property.

\subsection{Proof of Proposition \ref{prop:contdep-V}}
The proof follows the scheme of \cite[Proposition~6.3]{CGDP2010}. The necessary changes are only in the
estimate of the nonlinear difference term $\phi(u_2)-\phi(u_1)$ in the ${V^1}$-norm, in order to verify
assumption \eqref{eq:CGDP56-ass-updated} of Lemma~\ref{lem:CGDP55-56-updated}  for the
difference system.
Namely, it suffices to show that there exists
$\upsilon=\upsilon(R)$ such that for a.e.\ $t\ge 0$,
\begin{equation}\label{eq:g-Lip-target}
\|\phi(u_2(t))-\phi(u_1(t))\|_{{V^1}}\le \mathcal I(R)\|W(t)\|_{\mathcal V}.
\end{equation} 
Fix $p\in(1,p_0)$.
Since $u_j(t)\in {V^2}=D(A_\mu)$ and $\|u_j(t)\|_{{V^2}}\le \mathcal I(R)$ for $z_i\in B_{\mathcal V}(R)$ by \eqref{eq:V-smoothing-est-V0}, the elliptic estimate of
Proposition~\ref{prop:W2p-A}     gives for $j=1,2$
\[
\|u_j(t)\|_{W^{2,p}}\le C \|u_j(t)\|_{{V^2}}\le \mathcal I(R),
\qquad
\|w(t)\|_{W^{2,p}}\le C \|w(t)\|_{{V^2}}\le C\|W(t)\|_{\mathcal V}.
\]
Recalling the embedding $W^{2,p}(\Omega)\hookrightarrow L^\infty(\Omega)$, we also have  
\begin{equation}\label{eq:Linfty-from-V}
\|u_j(t)\|_{\infty}\le \mathcal I(R),
\qquad
\|w(t)\|_{\infty}\le C\|W(t)\|_{\mathcal V}.
\end{equation}
By the chain rule, 
\[
\|\phi(u_2)-\phi(u_1)\|_{{V^1}}\leq C \|\phi(u_2)-\phi(u_1)\|_{V^0}
+ C\|\nabla(\phi(u_2)-\phi(u_1))\|_{V^0}.
\]
Using the mean value theorem pointwise,
\[
\phi(u_2)-\phi(u_1)=\phi'(\theta)w,\qquad 
\phi'(u_2)-\phi'(u_1)=\phi''(\beta)w,
\]
for intermediate functions $\theta(x,t),\beta(x,t)$ between $u_1(x,t)$, $u_2(x,t)$.
Therefore, relying on \textbf{(P2)}
\[
\|\phi(u_2)-\phi(u_1)\|_{V^0}\le \|\phi'(\theta)\|_{\infty}\|w\|_{V^0} \leq C\left(1+ \|u_1\|_\infty^{m} + \|u_2\|_\infty^{m}\right) \|w\|_{V^0}  \leq \mathcal I (R) \|W\|_{\mathcal V}.
\]
Similarly
\begin{equation}\label{eq:grad-diff-phi}
\begin{split}
 &\quad \|\nabla(\phi(u_2)-\phi(u_1))\|_{V^0}
\le \|\phi''(\beta)\|_{\infty}\|w\|_{\infty}\|\nabla u_2\|_{V^0}
+\|\phi'(u_1)\|_{\infty}\|\nabla w\|_{V^0}
\\ & \leq C\left(1+ \|u_1\|_\infty^{m-1} + \|u_2\|_\infty^{m-1}\right) \|w\|_{\infty}\|\nabla u_2\|_{V^0} +  C\left(1+ \|u_1\|_\infty^{m}  \right)\|\nabla w\|_{V^0}\\ &\leq \mathcal I (R) \|W\|_{\mathcal V}.
\end{split}
\end{equation}
Combining the  last two displayed bounds  yields \eqref{eq:g-Lip-target}, so that 
  Lemma~\ref{lem:CGDP55-56-updated} \textup{(B)} applies    to the difference system, whence \eqref{eq:contdep-V}.

\section{Regular exponentially attracting sets}\label{sec:reg-exp-attr}
This section serves the purpose of showing that the semigroup generated by Problem \eqref{eq:CG-QC} on the regular space $\mathcal V$ is exponentially attracted, up to subtraction of the solution of the linear stationary problem, by a ball in a more regular space which is compactly embedded into $\mathcal V$, and that we proceed to define below.

Throughout, we refer to the $\mathcal V$--absorbing ball $\mathbb B_{\mathcal V}=B_{\mathcal V}(R_{\mathcal V})$
from Proposition~\ref{prop:V-smoothing}.
Recall the vector $z_f$ defined in \eqref{eq:main-zf}.  
Since $f\in V^0$, we have $u_f\in {V^2}=D(A_\mu)$ and therefore $z_f\in\mathcal V$.
Moreover, the normalization \eqref{e:kappanorm} yields
\[\int_0^\infty h(s)A_\mu\eta_f(s)\,\d s=A_\mu u_f,\] hence $z_f$ is a stationary solution of Problem \eqref{eq:CG-QC}.
With the aim of achieving compactness of the memory component, introduce the  functional $\Xi:\mathcal M^2\to \mathbb R$ 
\begin{equation}\label{eq:Xi-h-def}
\Xi[\eta]\coloneqq \|T\eta\|_{\mathcal M^2}^2
+\sup_{x\ge1} x\int_{(0,x^{-1})\cup(x,\infty)}  h(s)\|\eta(s)\|_{{V^2}}^2\,\d s,
\end{equation}
and define
\begin{equation}\label{eq:Wh-def}
\mathcal W\coloneqq \left\{\eta\in \mathcal M^4\cap \mathcal D(T_2):\Xi[\eta]<\infty\right\},
\qquad
\|\eta\|_{\mathcal W}^2\coloneqq \|\eta\|_{\mathcal M^4}^2+\Xi[\eta].
\end{equation}
The compactifying phase space we need is then
\begin{equation}\label{eq:Z-def}
\mathcal Z\coloneqq {V^3}\times \mathcal W,
\qquad
\|(u,\eta)\|_{\mathcal Z}^2\coloneqq \|u\|_{{V^3}}^2+\|\eta\|_{\mathcal W}^2.
\end{equation}

\begin{remark}\label{rem:Z-compact} Indeed,
the embedding ${V^3}\Subset {V^2}$ is compact by spectral theory, relying on the compact resolvent of $A_\mu$.
The choice of $\Xi$ is designed so that bounded sets of $\mathcal W$ are relatively compact in
$\mathcal M^2$. Consequently,
$
\mathcal Z\Subset \mathcal V.
$
The proof is identical to \cite[Remark~7.1]{CGDP2010}, with $h$ in place of $\mu$, and we omit it.
\end{remark} With these definitions in place, we state the main result of this section.
\begin{proposition}[Regular exponentially attracting set]\label{thm:regular-exp-attract-sec7}
There exist $R_\star>0$ and $t_\star>0$, depending only on the structural parameters, such that the set
\begin{equation}\label{eq:B-def-sec7}
\mathbb S_{\mathcal Z}\coloneqq z_f+\mathbb B_{\mathcal Z}(R_\star)
\end{equation}
satisfies:
\begin{itemize}
\item[\textup{(i)}]  
$S(t) \mathbb S_{\mathcal Z}\subset \mathbb S_{\mathcal Z}$ for all $t\ge t_\star$, 
\item[\textup{(ii)}]  
there exist structural constants $C\ge1$ and $\varepsilon>0$ such that
\begin{equation}\label{eq:exp-attract-B-sec7}
\dist_{\mathcal V}\left(S(t)\mathbb B_{\mathcal V},\mathbb S_{\mathcal Z}\right)\le C\e^{-\varepsilon t},
\qquad \forall t\ge0.
\end{equation}
\end{itemize}
\end{proposition}
The proof rests on a tailored decomposition of the semigroup $ S(t)$ on $\mathcal V$ akin to that of \cite[Section 7]{CGDP2010}.  We postpone it until after the main difference, which is the suitable estimation of the nonlinear term in the regular space ${V^2}$. This estimate is strongly affected by the variable and non-smooth coefficient setting we are placed in by the Beltrami operator.  \begin{lemma}[Nonlinear estimate at the graph level ${V^2}$]\label{lem:phi-in-H2}
Assume \eqref{e:ellipticitymu}, and in addition that \eqref{e:regmu} for some $p_0\ge2$.
There holds
\begin{equation}\label{eq:phi-H2-est-subbed}
\|\phi(u)\|_{{V^2}}
\le C\left(1+\|u\|_{{V^2}}^{m+1}\right).
\end{equation}
\end{lemma}

\begin{proof} Fix $p  \in (\frac43,p_0)$.
Applying Proposition~\ref{prop:W2p-A} with $g=A_\mu u$, and taking advantage of spectral calculus, \eqref{eq:uniform-coercivity}, and standard embeddings on bounded domains,
\begin{equation}\label{eq:W2p-by-H2}
\|u\|_{W^{2,p}}
\le C\left(\|A_\mu u\|_{L^{p}}+\|u\|_{V^0}\right)
\le C\|u\|_{{V^2}}.
\end{equation}
  Sobolev embeddings in dimension $2$ then yield
\begin{equation}\label{eq:Linfty-and-L4-by-W2p}
\|u\|_{\infty}+\|\nabla u\|_{L^4(\Omega)}
\le C\|u\|_{W^{2,p}}
\le C\|u\|_{{V^2}}.
\end{equation}
Combining    \eqref{eq:Linfty-and-L4-by-W2p} with  assumption (P2),
\begin{equation}\label{eq:phi-derivs-subbed}
\|\phi'(u)\|_{\infty}\le  C\left(1+\|u\|_{{V^2}}^m\right),
\qquad
\|\phi''(u)\|_{\infty}\le   C\left(1+\|u\|_{{V^2}}^{m-1}\right).
\end{equation}
We now use the chain rule at the level of $A_\mu$. Namely,
set \[F\coloneqq \sigma_\mu\nabla u.\] Since $u\in {V^2}$, we have $F\in L^2(\Omega;\R^2)$ and
$\mathrm{div} F=-A_\mu u\in L^2(\Omega)$, hence $F\in H(\mathrm{div};\Omega)$. Also, by Step~1 and Step~2,
$\phi'(u)\in L^\infty$ and $\nabla(\phi'(u))=\phi''(u)\nabla u\in L^4$, so
$\phi''(u)\,(\sigma_\mu\nabla u)\cdot \nabla u \in L^2$.
The standard product rule in distributions, justified by Galerkin approximation, yields
\[
A_\mu(\phi(u))
=-\mathrm{div}\left(\sigma_\mu\nabla(\phi(u))\right)
=-\mathrm{div}\left(\phi'(u)F\right)
=-\phi'(u)\mathrm{div} F-\nabla(\phi'(u))\cdot F,
\]
hence, using $\mathrm{div} F=-A_\mu u$ and $\nabla(\phi'(u))=\phi''(u)\nabla u$,
\[
A_\mu(\phi(u))=\phi'(u)\,A_\mu u-\phi''(u)\,(\sigma_\mu\nabla u)\cdot\nabla u.
\]
Therefore, by \eqref{eq:ellipticity},
\begin{equation}\label{eq:phi-H2-core}
\|\phi(u)\|_{{V^2}}=\|A_\mu(\phi(u))\|_{V^0}
\le \|\phi'(u)\|_{\infty}\|A_\mu u\|_{V^0}
+ C(k)\|\phi''(u)\|_{\infty}\|\nabla u\|_{L^4(\Omega)}^2.
\end{equation}
Substituting \eqref{eq:phi-derivs-subbed} and \eqref{eq:Linfty-and-L4-by-W2p} into \eqref{eq:phi-H2-core} gives
\[
\|\phi(u)\|_{{V^2}}
\le C\left[(1+\|u\|_{{V^2}}^m)\|u\|_{{V^2}}+(1+\|u\|_{{V^2}}^{m-1})\|u\|_{{V^2}}^2\right]
\le C\left(1+\|u\|_{{V^2}}^{m+1}\right),
\]
which is \eqref{eq:phi-H2-est-subbed}.
\end{proof}
We now turn to the proof of Proposition \ref{thm:regular-exp-attract-sec7}. Let $R>0$, fix  $z\in \mathbb B_{\mathcal V}(R)$ and indicate $S(t)z=(u(t),\eta^t)$.   The upshot is  the decomposition
\begin{equation}\label{eq:decomp-sec7}
S(t)z = z_f+\ell_1(t;z)+\ell_2(t;z), \qquad \ell_1(\cdot;z):[0,\infty)\to\mathcal V,
\quad
\ell_2(\cdot;z):[0,\infty)\to\mathcal Z,
\end{equation}
with the property that
\begin{equation}\label{eq:ell1-decay}
\|\ell_1(t;z)\|_{\mathcal V}\le \mathcal I(R)\e^{-c t},\qquad 
 t\ge0,
\end{equation}
and
\begin{equation}\label{eq:ell2-Z-bdd}
\sup_{t\ge0}\ \|\ell_2(t;z)\|_{\mathcal Z}\le \mathcal I(R).
\end{equation}
Moreover, referring to \eqref{eq:main-zf}, if $z-z_f\in \mathbb B_{\mathcal Z}(\varrho)$ for some $\varrho>0$, then
\begin{equation}\label{eq:ell1-Z-est}
\|\ell_1(t;z)\|_{\mathcal Z}\le \mathcal I(\varrho)\e^{-c t}+\mathcal I(\varrho), \qquad
t\ge0.
\end{equation}
First, we define $\ell_1,\ell_2$.
Let $L(t)$ denote the linear homogeneous semigroup associated with
\eqref{eq:CG-QC} with $\phi\equiv0$ and $f\equiv0$.
Set
\begin{equation}\label{eq:ell1-ell2-def}
\ell_1(t;z)\coloneqq L(t)\left(z-z_f\right).
\end{equation}
Taking differences, $\ell_2(t;z)$ coincides with    the unique solution $W(t)=(w(t),\xi^t)$ of the forced linear system
\begin{equation}\label{eq:W-system-sec7}
\left\{
\begin{aligned}
&\partial_t w + A_{\mu}w + \int_0^\infty h(s)A_\mu\xi(s)\,\d s = -\phi(u(t)),\\
&\partial_t \xi = T\xi + w,\\
&W(0)=(0,0),
\end{aligned}
\right.
\end{equation}
The exponential decay of the linear homogeneous semigroup $L(t)$ on $\mathcal V$ can be read directly from
 \cite[Proposition~6.5]{CGDP2010}, as it rests on the abstract properties of the Hilbert scale $V^r$ and on the memory kernel assumptions, which are equivalent to those in the quoted proposition.
 
   We move on to regularity of $\ell_2$.
The energy reconstruction at the $\mathcal Z$--level and the control of the corresponding  history functional $\Xi$
are identical to \cite[Step~2 of Lemma~7.3 together with Lemma~7.4]{CGDP2010} and will not be repeated  once  the boundedness of the forcing
$g(t)=-\phi(u(t))$ in ${V^2}$ on the $\mathcal V$--absorbing set is verified.
Recalling Proposition \ref{prop:main-H1-link}, as $z\in\mathbb B_{\mathcal V}(R)$  we have
\begin{equation}\label{eq:u-H2-unif-sec7}
\sup_{t\ge0}\ \|u(t)\|_{{V^2}}\le \mathcal I(R).
\end{equation}
Applying Lemma~\ref{lem:phi-in-H2} pointwise in time then  yields
\begin{equation}\label{eq:phi-H2-unif-sec7}
\sup_{t\ge0}\ \|\phi(u(t))\|_{{V^2}}\le \mathcal I(R).
\end{equation}
With \eqref{eq:phi-H2-unif-sec7} available, the argument of \cite[Lemma~7.3]{CGDP2010} applies verbatim to
\eqref{eq:W-system-sec7} and yields \eqref{eq:ell2-Z-bdd}. The refined estimate \eqref{eq:ell1-Z-est}
follows from the same functional bound on $\Xi$ as in \cite[Lemma~7.4]{CGDP2010}.
This completes the proof of Proposition~\ref{thm:regular-exp-attract-sec7}.

\section{Proof of Theorem \ref{thm:main-exp-attractor-V}}\label{sec:exp-attractor-V}
This section completes the proof of Theorem \ref{thm:main-exp-attractor-V} and of Corollary \ref{cor:main-global-attractors}. 

\subsection{Proof of Theorem \ref{thm:main-exp-attractor-V}, points (i)--(iv)}
 The first four steps of the proof are the actual construction of the set $\mathfrak E$ as a regular exponential attractor for the semigroup generated by \eqref{eq:CG-QC} on $\mathcal V$, cf.\ Proposition~\ref{prop:V-smoothing}.

Recall that we have constructed a  regular exponentially attracting set $\mathbb S_{\mathcal Z}$ in
Theorem~\ref{thm:regular-exp-attract-sec7}, and the compact embedding $\mathcal Z\Subset \mathcal V$,
see Remark~\ref{rem:Z-compact}.
By transitivity of exponential attraction \cite{FGMZ}, since the absorbing ball $\mathbb B_{\mathcal V}$ is exponentially
attracted by $\mathbb S_{\mathcal Z}$ in $\mathcal V$, cf.\ Theorem~\ref{thm:regular-exp-attract-sec7},
it suffices to construct an exponential attractor for the restriction of the dynamics to  the set $\mathbb S_{\mathcal Z}$ given by
\begin{equation}\label{eq:B-set-sec8}
 \mathbb S_{\mathcal Z}=z_f+\mathbb B_{\mathcal Z}(R_\star)
\end{equation} for some $R_\star>0$ depending only on structural parameters. Theorem~\ref{thm:regular-exp-attract-sec7} also entails the existence of  a structural $t_\star>0$ with the property \begin{equation}\label{eq:Sinvariance8}
 S(t)\mathbb S_{\mathcal Z} \subset  \mathbb S_{\mathcal Z} , \qquad t\geq t_\star,\end{equation}
Combining the  main result of \cite{EMZ} on exponential attractors of discrete semigroups with the standard discrete-to-continuous argument, see e.g.\ \cite[Section 8]{CGDP2010}, the existence of a set $\mathfrak E$ satisfying (i), (iii) and (iv) of Theorem \ref{thm:main-exp-attractor-V} follows from the verification of the two properties \textbf{(S)} and \textbf{(T)} below. \smallskip
\noindent\textbf{(S) Squeezing property}
There exist  structural constants $\vartheta\in(0,1)$ and $C>0$ such that for all $z_1,z_2\in  \mathbb S_{\mathcal Z}$
there exist decompositions
\begin{equation}\label{eq:squeeze-decomp-sec8}
\begin{split}
&S(t_\star)z_1-S(t_\star)z_2=\ell_1+\ell_2,
\\
&
\|\ell_1\|_{\mathcal V}\le \vartheta\|z_1-z_2\|_{\mathcal V},
\qquad
\|\ell_2\|_{\mathcal Z}\le C\|z_1-z_2\|_{\mathcal V}.
\end{split}
\end{equation}

\smallskip
\noindent\textbf{(T) Lipschitz continuity in time}
There exists $C>0$ such that for all $t,\tau\in[t_\star,2t_\star]$,
\begin{equation}\label{eq:time-Lip-sec8}
\sup_{z\in \mathbb S_{\mathcal Z}}\ \|S(t)z-S(\tau)z\|_{\mathcal V}\le C|t-\tau|.
\end{equation}

With (S) and (T) in hand,  one constructs an exponential attractor $\mathfrak E_d\subset  \mathbb S_{\mathcal Z}$ for the discrete map
$S(t_\star): \mathbb S_{\mathcal Z}\to  \mathbb S_{\mathcal Z}$, and then sets
\begin{equation}\label{eq:E-continuous-sec8}
\mathfrak E\coloneqq \bigcup_{t\in[t_\star,2t_\star]} S(t)\mathfrak E_d.
\end{equation}
Property (T) ensures that the map $(t,z)\mapsto S(t)z$ is Lipschitz from
$[t_\star,2t_\star]\times B$ into $\mathcal V$, hence $\mathfrak E$ is compact in $\mathcal V$
and has finite $\mathcal V$-fractal dimension there. Positive invariance is inherited from  positive invariance of $  \mathbb S_{\mathcal Z}$
and \eqref{eq:E-continuous-sec8}. The exponential attraction in $\mathcal V$ then follows by transitivity.

The proof of (S)-(T) can be obtained along the lines of \cite[\S8]{CGDP2010} once  the estimates for the differences of the   nonlinearities appearing therein are replaced by suitable bounds compatible with the two-dimensional and rough coefficient setting.
 
\begin{lemma}[Local Lipschitz of $\phi$ in ${V^2}$]\label{lem:phi-lip-H2}
Assume   \eqref{e:regmu} holds for some $p_0 \geq 2$. Then 
\begin{equation}\label{eq:phi-lip-H2}
\|\phi(u_2)-\phi(u_1)\|_{{V^2}}\le \mathcal I(\|u_1\|_{{V^2}}+ \|u_2\|_{{V^2}}) \|u_2-u_1\|_{{V^2}}.
\end{equation}
\end{lemma}

\begin{proof}
Set $w\coloneqq u_2-u_1$, and fix $p\in (\frac43,p_0)$. As before, Sobolev embeddings and  the estimates of Proposition~\ref{prop:W2p-A}   
yield the estimate 
\begin{equation}\label{eq:Linfty-L4-sec8}
\|u\|_{\infty}+\|\nabla u\|_{L^4(\Omega)}\le C\|u\|_{W^{2,p}}\le C\|u\|_{{V^2}}
\end{equation}
Using \textbf{(P2)}, and \eqref{eq:Linfty-L4-sec8} we then have 
\begin{equation}\label{eq:phi-derivs-bdd-sec8}
\|\phi^{(1+j)}(\theta)\|_{\infty}\le \mathcal I(\|u_1\|_{{V^2}}+ \|u_2\|_{{V^2}}), \qquad j=0,1,2\end{equation}
for any intermediate function $\theta$ between $u_1$ and $u_2$.
The chain rule reads
\begin{equation}\label{eq:chain-Amu-sec8}
A_\mu(\phi(u))=\phi'(u)A_\mu u-\phi''(u)(\sigma_\mu\nabla u)\cdot\nabla u
\qquad\text{in }L^2(\Omega).
\end{equation}
Applying \eqref{eq:chain-Amu-sec8} to $u_2$ and $u_1$ and subtracting gives
\[
A_\mu(\phi(u_2)-\phi(u_1))=\mathrm{I}+\mathrm{II}-\mathrm{III},
\]
where
\[
\begin{split}&
\mathrm{I}\coloneqq \phi'(u_2)A_\mu w,
\qquad
\mathrm{II}\coloneqq \left(\phi'(u_2)-\phi'(u_1)\right)A_\mu u_1,
\\
&
\mathrm{III}\coloneqq
\phi''(u_2)(\sigma_\mu\nabla u_2)\cdot\nabla u_2
-\phi''(u_1)(\sigma_\mu\nabla u_1)\cdot\nabla u_1.
\end{split}
\]
We estimate these terms in $L^2$.

\smallskip
\noindent\emph{Term $\mathrm{I}$.} Using \eqref{eq:phi-derivs-bdd-sec8},
\[
\|\mathrm{I}\|_{V^0}\le \|\phi'(u_2)\|_{\infty}\|A_\mu w\|_{V^0}
\le  \mathcal I(\|u_1\|_{{V^2}}+ \|u_2\|_{{V^2}})\|w\|_{{V^2}}.
\]
\smallskip
\noindent\emph{Term $\mathrm{II}$.} By the mean value theorem,
$\phi'(u_2)-\phi'(u_1)=\phi''(\theta)w$ for an intermediate function $\theta$, hence a combination of \eqref{eq:Linfty-L4-sec8} and \eqref{eq:phi-derivs-bdd-sec8} entails
\[
\|\mathrm{II}\|_{V^0}
\le \|\phi''(\theta)\|_{\infty}\|w\|_{\infty}\|A_\mu u_1\|_{V^0}
\le \mathcal I (\|u_1\|_{{V^2}} + \|u_2\|_{{V^2}})\|w\|_{{V^2}}.
\]
\smallskip
\noindent\emph{Term $\mathrm{III}$.} Rewrite
\[
\mathrm{III}=
\phi''(u_2)\left(Q_2-Q_1\right)+\left(\phi''(u_2)-\phi''(u_1)\right)Q_1,
\qquad
Q_j\coloneqq (\sigma_\mu\nabla u_j)\cdot\nabla u_j, \quad j=1,2.
\]
For the first part, note that
\[
Q_2-Q_1=(\sigma_\mu\nabla w)\cdot\nabla u_2+(\sigma_\mu\nabla u_1)\cdot\nabla w,
\]
so, by ellipticity \eqref{eq:ellipticity} and H\"older's inequality
\[
\|\phi''(u_2)(Q_2-Q_1)\|_{V^0}
\le \|\phi''(u_2)\|_\infty\|\nabla w\|_{L^4(\Omega)}\left(\|\nabla u_2\|_{L^4(\Omega)}+\|\nabla u_1\|_{L^4(\Omega)}\right)
\le \mathcal I(\|u_1\|_{{V^2}}+ \|u_2\|_{{V^2}})\|w\|_{{V^2}},
\]
having relied upon \eqref{eq:Linfty-L4-sec8} and \eqref{eq:phi-derivs-bdd-sec8}.
For the second part, again by the mean value theorem,
$\phi''(u_2)-\phi''(u_1)=\phi'''(\zeta)w$ for an intermediate $\zeta$, hence
\[
\|(\phi''(u_2)-\phi''(u_1))Q_1\|_{V^0}
\le \|\phi'''(\theta)\|_\infty\|w\|_{\infty}\|Q_1\|_{V^0}.
\]
Finally, $\|Q_1\|_{V^0}\leq C \|\nabla u_1\|_{L^4(\Omega)}^2\le  \mathcal I(\|u_1\|_{{V^2}} )$ by \eqref{eq:Linfty-L4-sec8},
and it follows that the last display is $ \leq \mathcal I(\|u_1\|_{{V^2}}+ \|u_2\|_{{V^2}})\|w\|_{{V^2}}$ as well.
Collecting the estimates for $\mathrm{I},\mathrm{II},\mathrm{III}$ finally yields \eqref{eq:phi-lip-H2}.
\end{proof}

\begin{lemma} \label{lem:phi-prime-prod-H1}
Let $z\in \mathbb S_{\mathcal Z}$ be arbitrary and write $  S(t)z=(u(t),\eta^t).$ For each $v\in {V^2}$,
\begin{equation}\label{eq:phi-prime-prod-H1}
\|\phi'(u(t))v\|_{{V^1}}\le C\|v\|_{{V^2}}, \qquad t\geq t_\star.
\end{equation}
\end{lemma}

\begin{proof} Recalling the structural constant $\mathcal V $-boundedness of $\mathbb S_{\mathcal Z}$, the positive invariance property \eqref{eq:Sinvariance8}, the estimate \eqref{eq:Linfty-L4-sec8} together with assumption (P2), we note that 
\begin{equation}
	\label{eq:S8uitb}
\|\phi'(u(t))\|_{\infty}+\|\phi''(u(t))\|_{\infty}
+ \|\nabla u(t)\|_{L^4(\Omega)}
\le C,    \qquad t\geq t_\star.
\end{equation}
Let $v\in {V^2}$. Ellipticity entails that  ${V^1}\simeq H^1_0(\Omega)$,
\[
\|\phi'(u)v\|_{{V^1}}\leq C\|\phi'(u)v\|_{V^0}+ C\|\nabla(\phi'(u)v)\|_{V^0}.
\]
We estimate
\[
\|\phi'(u)v\|_{V^0}\le \|\phi'(u)\|_{\infty}\|v\|_{V^0}\le C\|v\|_{{V^2}}.
\]
Moreover,
\[
\nabla(\phi'(u)v)=v\phi''(u)\nabla u+\phi'(u)\nabla v,
\]
hence, by H\"older and \eqref{eq:S8uitb}, \[
\|\nabla(\phi'(u)v)\|_{V^0}
\le \|\phi''(u)\|_\infty\|v\|_{\infty}\|\nabla u\|_{V^0}
+\|\phi'(u)\|_\infty\,\|\nabla v\|_{V^0}
\le C\left(\|v\|_{\infty}+\|\nabla v\|_{V^0}\right).
\]
Finally, by Proposition~\ref{prop:W2p-A} and Sobolev embedding, as before
$\|v\|_{\infty}+\|\nabla v\|_{V^0}\le C\|v\|_{{V^2}}$,
and \eqref{eq:phi-prime-prod-H1} follows.
\end{proof}


We now indicate how Lemmas~\ref{lem:phi-lip-H2}--\ref{lem:phi-prime-prod-H1} enter the verification of
\textbf{(S)}--\textbf{(T)}. 

\subsubsection*{Verifying \textbf{\emph{(S)}}}
Fix $z_1,z_2\in \mathbb S_{\mathcal Z}$ and set
\[
S(t)z_i=(u_i(t),\eta_i^t),\qquad W(t)\coloneqq S(t)z_2-S(t)z_1=(w(t),\bar\eta^t).
\]
As in \cite[\S8]{CGDP2010}, decompose $W(t_\star)$ as in \eqref{eq:squeeze-decomp-sec8} by writing the variation-of-constants
formula for the difference system and splitting the contribution into a decaying linear part and a regular part.
More precisely, let $L(t)$ denote the linear homogeneous semigroup generated by  \eqref{eq:CG-QCext} with $g\equiv 0$. Set
\[
\ell_1 \coloneqq L(t_\star)\left(z_2-z_1\right),
\qquad
\ell_2 \coloneqq W(t_\star)-\ell_1.
\]
The estimate $\|\ell_1\|_{\mathcal V}\le \vartheta\|z_2-z_1\|_{\mathcal V}$ for some $\vartheta\in(0,1)$ follows from
the exponential stability of $L(t)$ on $\mathcal V$, cf.\ \cite[Prop.~6.5]{CGDP2010}, which uses only the abstract Hilbert scale
associated with $A_\mu$ and the memory assumptions.

It remains to bound $\ell_2$ in $\mathcal Z$. The element $\ell_2$ corresponds to the solution, at time $t_\star$, of  \eqref{eq:CG-QCext} with $g=\phi(u_1)-\phi(u_2)$.   The $\mathcal Z$-bound
\[
\|\ell_2\|_{\mathcal Z}\le C\|z_2-z_1\|_{\mathcal V}
\]
is obtained as in \cite[\S8]{CGDP2010},  once
Lemma~\ref{lem:phi-lip-H2} is used to deduce that  
\[
\|\phi(u_2(t))-\phi(u_1(t))\|_{{V^2}}\le C\|w(t)\|_{{V^2}}\le C\|W(t)\|_{\mathcal V}, \qquad t\ge t_\star
\]
The remaining 
$\mathcal Z$--energy reconstruction and  the memory compactness functional estimates are identical to \cite[\S8]{CGDP2010} and are omitted.
Thus \textbf{(S)} holds.
\subsubsection*{Verifying \textbf{\emph{(T)}}}
Fix $z\in \mathbb S_{\mathcal Z}$ and write $S(t)z=(u(t),\eta^t)$. 
It suffices to show that \begin{equation}\label{eq:time-der-unif}
\sup_{t\in[t_\star,2t_\star]}\ \left\|(\pt u(t),\pt\eta^t)\right\|_{\mathcal V}\le C, 
\end{equation}
as the Lipschitz estimate \eqref{eq:time-Lip-sec8} then follows by integration in time.
Differentiating \eqref{eq:CG-QC} in time, we observe that after  time-shift $t\mapsto t_\star+t$, the pair
\[
U(t)\coloneqq \pt u(t+t_\star),\qquad \zeta(t)\coloneqq \pt\eta^{t+t_\star}
\]
satisfies 
\eqref{eq:aux-CG} with $f=0$ and forcing
\begin{equation}\label{eq:g-time-der}
g(t)=-\phi'(u(t_\star+t))U( t).
\end{equation}
 Since $S(t)\mathbb S_{\mathcal Z}\subset\mathbb S_{\mathcal Z}$ for $t\ge t_\star$, the bounds
defining $\mathbb S_{\mathcal Z}$ and Lemma~\ref{lem:phi-prime-prod-H1} imply
\[
\|g(t)\|_{{V^1}}\le C\|U(t)\|_{{V^2}}\qquad \forall\,t\ge t_\star.
\]
Applying the energy inequality \eqref{eq:F-damped} at level $r=1$ to $(U,\zeta)$ yields a
translation-bounded estimate of the form
\[
\sup_{t\ge t_\star}\int_t^{t+1}\|(U(\tau),\zeta^\tau)\|_{\mathcal V}^2\,\d\tau \le C,
\]
with $C$ depending only on the structural parameters and $R_\star$.
Hence Lemma~\ref{lem:CGDP55-56-updated}\textup{(A)} applied to $(U,\zeta)$ yields \eqref{eq:time-der-unif} and completes the proof of (i), (iii) and (iv) of Theorem \ref{thm:main-exp-attractor-V}. 

Point (ii) is easily seen as a consequence of the other properties. Indeed,  observe that  $\mathfrak E$ satisfies
\[
\mathfrak E\subset \mathbb S_{\mathcal Z}=z_f+\mathbb B_{\mathcal Z}(R_\star),
\]
hence $\mathfrak E-z_f$ is bounded in $\mathcal Z={V^3}\times \mathcal W$, proving in particular the ${V^3}$-boundedness.
Moreover, under \eqref{e:regmu} with $p_0=2$, Proposition~\ref{prop:W2p-A} yields the continuous embedding
${V^2}=D(A_\mu)\hookrightarrow W^{2,p}(\Omega)$ for every $1<p<2$; therefore the identity map
\[
\mathcal V={V^2}\times \mathcal M^2 \ \longrightarrow\ W^{2,p}(\Omega)\times \mathcal H^2
\]
is continuous, and since $\mathfrak E$ has finite fractal dimension in $\mathcal V$ by (i),
it has finite fractal dimension in $W^{2,p}(\Omega)\times \mathcal H^2$ as well. This proves (ii).

 \subsection{Proof of Theorem~\ref{thm:main-exp-attractor-V}, points (v), (vi)}
 
Let $R>0$ and write $B_R\coloneqq B_{\mathcal H^1}(R)$. By Proposition~\ref{prop:V-smoothing},
\[
S(1)B_R \subset B_{\mathcal V}(R_1), \qquad R_1=\mathcal I(R).
\]
Hence for $t\ge 1$, using point \textup{(iv)}  and the continuous embedding
$\mathcal V\hookrightarrow \mathcal H^1$,
\[
\dist_{\mathcal V}(S(t)B_R,\mathfrak E)
=\dist_{\mathcal V}(S(t-1)S(1)B_R,\mathfrak E)
\le C\,\dist_{\mathcal V}(S(t-1)B_{\mathcal V}(R_1),\mathfrak E)
\le \mathcal I(R)\e^{-ct}.
\]
On the other hand for $t\leq 1 $, Propositions \ref{prop:main-H1-link} and  \ref{prop:V-smoothing} tell us exactly that
\[
\dist_{\mathcal H^1}(S(t)B_R,\mathfrak E) +
\sqrt{t}\dist_{\mathcal V}(S(t)B_R,\mathfrak E) \leq \mathcal I(R)
\]
Combining the last two displays yields point \textup{(v)}.
For point (vi), 
let $R>0$ and set $B^0_R\coloneqq B_{\mathcal H^0}(R)$. By Proposition~\ref{prop:main-H1-link},
\[
\dist_{\mathcal H^0}(S(t)B^0_R,\mathbb B_1)\le \mathcal I(R)\e^{-ct}, \qquad t\ge 0,
\]
for the $\mathcal H^1$-absorbing ball $\mathbb B_1$.
Since $\mathbb B_1\subset \mathcal H^1$, applying point \textup{(v)} to $\mathbb B_1$ and using transitivity of $\mathcal H^0$ exponential attraction \cite{EMZ}, which is justified by using the standard exponential-in-time continuous dependence estimate for \eqref{eq:CG-QC} in the $\mathcal H^0$-norm.
The latter follows from the coercivity/monotonicity structure in \textbf{(P1)} and does not require additional elliptic regularity.
This argument is the same as in \cite{AMN} for the Euclidean analogue and is thus omitted.

\subsection{Proof of Corollary~\ref{cor:main-global-attractors}}
Let $\mathfrak E\subset\mathcal V$ be the exponential attractor provided by
Theorem~\ref{thm:main-exp-attractor-V}. In particular, $\mathfrak E$ is compact in $\mathcal V$ and
positively invariant. Then  $\mathcal A_j$, the $\omega$--limit set of $ \mathfrak E$ in $\mathcal H^j$ for $j=0,1$, is the $\mathcal H^j$ global attractor for $S(t)$. As $\mathcal A_0 $ is contained in $\mathfrak E$ which is bounded in $\mathcal H^1$, it must be attracted by $\mathcal A_1$. Invariance then forces $\mathcal A_0\subset \mathcal A_1$. The obvious converse inclusion yields the equality $\mathcal A_0=\mathcal A_1 \eqqcolon \mathcal A$ and completes this part of the proof. The fact that $S(t)$ restricts on $\mathcal A$ to a group of operators is a standard consequence of the backwards uniqueness property enjoyed by differential systems with memory. This is observed in \cite{AMN} for the Euclidean analogue,  and is preserved for the semigroup generated by \eqref{eq:CG-QC} as well. The proof of the corollary is thus complete.
 \bibliography{BeltramiMemory}
\bibliographystyle{amsplain}
\end{document}